\documentclass[12pt,fleqn]{article}

\usepackage{csquotes}
\usepackage{mathtools,amsfonts,amsmath, amsthm,amssymb}
\def\bp{\begin{proof}}
\def\ep{\end{proof}}

\def\ho{\hookrightarrow}

\def\curl{\operatorname{curl}}
\usepackage{pdfsync}

\usepackage{color}

\definecolor{gr}{rgb}   {0.,   0.8,   0. } 
\definecolor{bl}{rgb}   {0.,   0.5,   1. } 
\definecolor{mg}{rgb}   {0.7,  0.,    0.7} 

\def\XXint#1#2#3{{\setbox0=\hbox{$#1{#2#3}{\int}$}
     \vcenter{\hbox{$#2#3$}}\kern-.5\wd0}}

\newcommand{\Bk}{\color{black}}

 \usepackage[utf8]{inputenc}
\usepackage[T1]{fontenc}
 \usepackage{fouriernc}
\usepackage{mathrsfs} 
\usepackage{stmaryrd}
\usepackage{bbding}
 \usepackage{latexsym}
\usepackage{esint}

\usepackage[margin=35mm]{geometry} 
 \numberwithin{equation}{section}

\def\VMO{\text{\rm VMO}}

\usepackage{color}
 \definecolor{db}{rgb}{0.0,0.0,0.8} 
\definecolor{dg}{rgb}{0.0,0.55,0.14}
\definecolor{dr}{rgb}{0.5,0,0.07}

\def\ho{\hookrightarrow}

  \usepackage{hyperref}
 \hypersetup{frenchlinks=true,colorlinks=true,linkcolor=db,citecolor=dg,bookmarks=true}

 \usepackage{enumerate}

\theoremstyle{definition}
\newtheorem{case}{Case}
 \swapnumbers\newtheorem{theorem}{Theorem}[section]
 \swapnumbers\newtheorem{proposition}[theorem]{Proposition}
 \swapnumbers\newtheorem{lemma}[theorem]{Lemma}
 \swapnumbers\newtheorem{corollary}[theorem]{Corollary}
\theoremstyle{definition}
 \swapnumbers
\theoremstyle{definition}
 \swapnumbers
 \theoremstyle{definition}
 \swapnumbers
\theoremstyle{definition}

\theoremstyle{definition}
\newtheorem{definition}[theorem]{Definition}
\theoremstyle{definition}
\newtheorem{remark}[theorem]{Remark}
\theoremstyle{definition}

\newcounter{step}

\def\be{\begin{equation}}
\def\ee{\end{equation}}
\def\bes{\begin{equation*}}
\def\ees{\end{equation*}}
\def\bt{\begin{theorem}}
\def\et{\end{theorem}}
\def\bpr{\begin{proposition}}
\def\epr{\end{proposition}}
\def\bl{\begin{lemma}}
\def\el{\end{lemma}}
\def\bc{\begin{corollary}}
\def\ec{\end{corollary}}
\def\br{\begin{remark}}
\def\er{\end{remark}}
\def\ben{\begin{enumerate}}
\def\bena{\begin{enumerate}[a)]}
\def\een{\end{enumerate}}
\def\bit{\begin{itemize}}
\def\iit{\end{itemize}}

\def\supp{\operatorname{supp}}

\def\tr{\operatorname{tr}}

\def\div{\operatorname{div}}

\DeclareMathAlphabet{\mathonebb}{U}{bbold}{m}{n}
\newcommand{\one}{\ensuremath{\mathonebb{1}}}
\def\R{{\mathbb R}}
\def\N{{\mathbb N}}

\def\C{{\mathbb C}}
\def\Z{{\mathbb Z}}

 \def\T{{\mathbb T}}

\def\la{\langle}
\def\ra{\rangle}

\def\fo{\forall\, }
\def\iff{\Longleftrightarrow}
\def\va{\varphi}
\def\d{\displaystyle}
\def\im{\imath}
\def\ve{\varepsilon}
\def\p{\partial}
\def\l{\label}

\def\na{\nabla}

\def\so{{\mathbb S}^1}

\date{\today}
\title{Lifting in Besov spaces}
\author{
Petru Mironescu    \thanks{Universit\'e de Lyon,  CNRS UMR 5208, Universit\'e Lyon 1, Institut Camille Jordan, 43 blvd. du 11 novembre 1918, F-69622 Villeurbanne cedex, France. Email address: mironescu$@$math.univ-lyon1.fr}         \and
Emmanuel Russ \thanks{Universit\'e Grenoble Alpes, CNRS UMR 5582, 100 rue des math\'ematiques, 38610 Gieres, France. Email address: emmanuel.russ$@$univ-grenoble-alpes.fr}
\and
Yannick Sire\thanks{Johns Hopkins University, Krieger Hall, Baltimore MD, USA. Email address: sire$@$math.jhu.edu}
}

\begin{document}

\maketitle

\section{Introduction}
${}$

Let $\Omega\subset \R^n$ be a bounded simply connected domain and $u:\Omega\rightarrow \so$ a continuous (resp. $C^k$, $k\geq 1$) function. It is a well-known fact that there exists a continuous (resp. $C^k$) real-valued function $\va$ such that $u=e^{\im\va}$. In other words, $u$ has a continuous (resp. $C^k$) lifting. \par
\noindent The analogous problem when $u$ belongs to the fractional Sobolev space $W^{s,p}$, $s>0$, $1\le p<\infty$, received an complete answer in \cite{lss}. Let us briefly recall the results:
\begin{enumerate}
\item when $n=1$, $u$ has a lifting in $W^{s,p}$ for all $s>0$ and all $p\in [1,\infty)$,
\item when $n\geq 2$ and $0<s<1$, $u$ has a lifting in $W^{s,p}$ if and only if $sp<1$ or $sp\geq n$,
\item when $n\ge 2$ and  $s\geq 1$, $u$ has a lifting in $W^{s,p}$ if and only if $sp\geq 2$.
\end{enumerate}

 Further developments in the Sobolev context can be found in \cite{bethuelchiron,nguyenphase,mironescuphase,mironescucras2}. 

 In the present paper, we address the corresponding question in the framework of Besov spaces. More specifically, given $s, p, q$ in suitable ranges defined later, we ask whether a map $u\in B^s_{p,q}(\Omega ; \so)$ can be lifted as $u=e^{\im\va}$, with $\va\in B^s_{p,q}(\Omega ; \mathbb R)$. We say that $B^s_{p,q}$ has the lifting property if and only if the answer is positive.

When dealing with circle-valued functions and their phases, it is natural to consider only maps in $L^1_{loc}$. This is why we assume that $s>0$,\footnote{ However, we will discuss an appropriate version of the lifting problem when $s\le 0$; see Remark \ref{aa1} and Case \ref{T} below. } and we take
the exponents $p$ and $q$  in the classical range $p\in [1,\infty)$, $q\in [1,\infty]$.\footnote{ We discard the uninteresting case where  $p=\infty$. In that case, maps in $B^s_{\infty,q}$ are continuous. Arguing as in Case \ref{tri} below, we obtain the existence of a $B^s_{\infty,q}$ phase for every $u\in B^s_{\infty,q}(\Omega ; \so)$.}

 Since Besov spaces are microscopic modifications of Sobolev (or Slobodeskii) spaces, one expects a global picture similar to the one described before for Sobolev spaces. The analysis in Besov spaces  is indeed partly similar to the one in Sobolev spaces, as far as the results and the techniques are concerned. However, several difficulties occur and some cases still remain open. Thus, the analysis of the lifting problem leads us to prove several new properties for Besov spaces (in connection with restriction or absence of restriction properties, sums of integer valued functions which are constant, products of functions in Besov spaces, disintegration properties for the Jacobian), which are interesting in their own right. We also provide detailed arguments for classical properties (some embeddings, Poincar\'e inequalities) which could not be precisely located in the literature.
 
\medskip
 Let us now describe more precisely our results and methods. When $sp>n$, functions in $B^s_{p,q}$ are continuous, which readily implies that $B^s_{p,q}$ has the lifting property (Case \ref{tri}).
 
\medskip
 In the case where $sp<1$, we rely on a characterization of $B^s_{p,q}$ in terms of the Haar basis \cite[Th\'eor\`eme 5]{bourdaud}, to prove that $B^s_{p,q}$ has the lifting property (Case \ref{A}).
 
\medskip
Assume now that $0<s\leq 1$, $sp=n$ and $q<\infty$.  Let $u\in B^s_{p,q}(\Omega ; \so)$ and let $F(x,\ve):=u\ast\rho_{\ve}$, where $\rho$ is a standard mollifier. Since $B^s_{p,q}\hookrightarrow \VMO$,  for all $\ve$ sufficiently small and  all $x\in \Omega$ we have $1/2<\left\vert F(x,\ve)\right\vert\le 1$. Writing $F(x,\ve)/\left\vert F(x,\ve)\right\vert=e^{\im\psi_\ve}$, where $\psi_\ve$ is $C^{\infty}$, and relying on a slight modification of the trace theory for weighted Sobolev spaces developed in \cite{tracesoldnew}, we conclude, letting $\varepsilon$tend to $0$, that $u=e^{\im\psi_0}$, where $\psi_0=\lim_{\ve\to 0}\psi_\ve\in B^s_{p,q}$, and therefore $B^s_{p,q}$ still has the lifting property (Case \ref{X}). 

\medskip
Consider now the case where $s>1$ and $sp\geq 2$. Arguing as in \cite[Section 3]{lss}, it is easily seen that the lifting property for $B^s_{p,q}$ will follow from the following property: given $u\in B^s_{p,q}(\Omega; \so)$, if $F:=u\wedge\na u\in L^p(\Omega ; \R^n)$, then $(*)$ $\curl F=0$. The proof of $(*)$ is much more involved than the corresponding one for $W^{s,p}$ spaces \cite[Section 3]{lss}. It relies on a disintegration argument for the Jacobians, more generally applicable in $W^{1/p,p}$. This argument, in turn, relies on the fact that $\curl F=0$ when $u\in \VMO$ and $n=2$, and a slicing argument. In particular, we need a {\it restriction property for Besov spaces}, namely the fact that, for $s>0$, $1\le p<\infty$ and $1\le q\le p$, for all $f\in B^s_{p,q}$, the partial maps of $f$ still belong to $B^s_{p,p}$ (see Lemma \ref{oa1} below). Thus, we obtain that, when $s>1$ and $1\leq p<\infty$, $B^s_{p,q}$ does have the lifting property when [$1\le q<\infty$, $n=2$, and $sp=2$], or [$1\le q\le p$, $n\ge 3$, and $sp= 2$], or [$1\le q\le \infty$, $n\ge 2$, and $sp> 2$].

 One can improve the conclusion of Lemma \ref{oa1} as follows.  For $s>0$, $1\le p<\infty$ and $1\le q\le p$, for all $f\in B^s_{p,q}$, the partial maps of $f$  belong to $B^s_{p,q}$ (Proposition \ref{qh1}). We emphasize the fact that this type of property holds only under the crucial assumption $q\le p$.
 More precisely, if $q>p$ and $s>0$, then we exhibit a compactly supported function $f\in B^s_{p,q}(\R^2)$ such that, for almost every $x\in (0,1)$, $f(x,\cdot)\notin B^s_{p,\infty}(\R)$ (Proposition \ref{l7.26}). This phenomenon, which has not been noticed before,  shows a picture strikingly different from the one for $W^{s,p}$, and even more generally for Triebel-Lizorkin spaces \cite[Section 2.5.13]{triebel2}.  
 
\medskip
 Let us return to the case when $0<s<1$, $1\le p< \infty$ and $n\geq 2$. Assume now that [$1\le q<\infty$ and $1\leq sp <n$], or [$q=\infty$ and $1< sp <n$]. In this case, we show that $B^s_{p,q}$ does not have the lifting property. The argument uses embedding theorems and the following fact, for which we provide a proof: let $s_i>0$, $1\le p_i<\infty$, and [$s_jp_j=1$ and $1\le q_j<\infty$], or [$s_jp_j>1$ and $1\le q_j\le\infty$], $i=1,2$.  Then, if $f_i\in B^{s_i}_{p_i,q_i}$ and $f_1+f_2$ only takes integer values, then the function $f_1+f_2$ is constant. 
 
 \medskip
Assume finally that $0<s<\infty$, $1\le p<\infty$, $n\ge 2$ and [$1\le q< \infty$ and $1\le sp<2$] or [$q=\infty$ and $1\le sp\le 2$]. In this case,  $B^s_{p,q}$ does not have the lifting property either. We provide a counterexample of topological nature, inspired by \cite[Section 4]{lss}: namely, the function $\d u(x)=\frac{(x_1,x_2)}{\d\left(x_1^2+x_2^2\right)^{1/2}}$ belongs to $B^s_{p,q}$ but has no lifting in $B^s_{p,q}$. 

\medskip
 Contrary to the case of Sobolev spaces, some cases remain open. A first case occurs when $s>1$, $1\le p<\infty$, $p<q<\infty$, $n\ge 3$, and $sp=2$. In this situation, since the restriction property for $B^s_{p,q}$ does not hold, the argument sketched before does not work any longer and we do not know if $B^s_{p,q}$ has the lifting property.
 
  The case where $s=1$, $1\le p<\infty$, $n\ge 3$, and [$1\le q<\infty$ and $2\le p<n$] or [$q=\infty$ and $2< p\le n$] is also open (except when $s=1$ and $p=q=2$, since in this case, $B^1_{2,2}=W^{1,2}$ has the lifting property). This is related to the fact that it is not known whether the map $\va\mapsto e^{\im\va}$ maps $B^1_{p,q}$ into itself.

 When $1\le p<\infty$, $s=1/p$ and $q=\infty$,  we do not know if $B^{1/p}_{p,\infty}$ has the lifting property. In particular, it is unclear whether the Haar system provides a basis of $B^{1/p}_{p,\infty}$. The case where $q=\infty$, $n\le  p<\infty$, $n\ge 3$ and $s=n/p$ is also open. Indeed, $B^s_{p,q}$ is not embedded into $\VMO$ in this case, and the argument briefly described above is not applicable any more.  
 
\medskip

 Let us summarize the main results of this paper concerning the lifting problem.
We start with positive cases.

\begin{theorem} \label{positive}
Let $s>0$, $1\le p<\infty$, $1\le q\le \infty$. The lifting problem has a positive answer in the following cases:
\begin{enumerate}
\item $s>0$, $1\le q\le\infty$, and $sp>n$,
\item $0<s<1$, $1\le q\le\infty$, and $sp<1$,
\item $0<s\le 1$, $1\le q<\infty$, and $sp=n$,
\item 
\begin{enumerate}
\item $s>1$, $1\le q<\infty$, $n=2$, and $sp=2$,
\item $s>1$, $1\le q\le p$, $n\ge 3$, and $sp= 2$,
\item $s>1$, $1\le q\le \infty$, $n\ge 2$,
and $sp> 2$.
\end{enumerate}
\end{enumerate}
\end{theorem}
\noindent The negative cases are as follows: 
\begin{theorem} \label{negative}
Let $s>0$, $1\le p<\infty$, $1\le q\le \infty$. The lifting problem has a negative answer in the following cases:
\begin{enumerate}
\item
\begin{enumerate}
\item $0<s<1$, $1\le q<\infty$, $n\ge 2$, and $1\leq sp <n$,
\item $0<s<1$, $q=\infty$, $n\ge 2$, and $1< sp <n$,
\end{enumerate}
\item
\begin{enumerate}
\item $0<s<\infty$, $1\le q<\infty$, $n\ge 2$, and $1\le sp<2$,
\item $0<s<\infty$, $1\le p<\infty$, $q=\infty$, $n\ge 2$, and $1<sp\le 2$.
\end{enumerate}
\end{enumerate}
\end{theorem}

The paper is organized as follows. In Section \ref{fun}, we briefly recall the standard definition of Besov spaces and some classical characterizations of these spaces (by Littlewood-Paley theory and wavelets). In Section \ref{pos} we establish Theorem \ref{positive}, namely the cases where $B^s_{p,q}$ does have the lifting property, while Section \ref{neg} is devoted to negative cases (Theorem \ref{negative}). In Section \ref{ope}, we discuss the remaining cases, which are widely open. The final section gathers statements and proofs of various results on Besov spaces needed in the proofs of Theorems \ref{positive} and \ref{negative}.

\subsubsection*{Acknowledgments}
 P. Mironescu  thanks N. Badr, G. Bourdaud, P. Bousquet, A.C.  Ponce and W. Sickel for useful discussions. He warmly thanks J. Kristensen for calling his attention to the reference \cite{uspenskii}. 
All the authors are  supported by the ANR project \enquote{Harmonic Analysis at its Boundaries},   ANR-12-BS01-0013-03. P. Mironescu was also supported by  the LABEX MILYON (ANR- 10-LABX-0070) of
Universit\'e de Lyon, within the program \enquote{Investissements d'Avenir}
(ANR-11-IDEX-0007) operated by the French National Research Agency
(ANR).

\section*{Notation, framework}
\ben
\item
Most of our results are stated in a smooth bounded domain $\Omega\subset\R^n$.
\item
In few cases, proofs are simpler if we consider $\Z^n$-periodic maps $u: (0,1)^n\to\so$. In this case, we denote the corresponding function spaces $B^s_{p,q}(\T^n ; \so)$, and the question is whether a map $u\in B^s_{p,q}(\T^n ; \so)$ has a  lifting $\va\in B^s_{p,q}((0,1)^n ; \R)$. [Of course, $\va$ need not be, in general, $\Z^n$-periodic.]  If such a $\va$ exists for every $u\in B^s_{p,q}(\T^n ; \so)$, then $B^s_{p,q}(\T^n ; \so)$ has the lifting property.

However, in these results it is not crucial to work in $\T^n$.  An inspection of the proofs shows that, with some extra work, 
 we could take any smooth bounded domain.
 \item
 In the same vein, it is sometimes easier to work in $\Omega =(0,1)^n$ (with no periodicity assumption). 
 \item
Partial derivatives are denoted $\p_j$, $\p_j\p_k$, and so on, or $\p^\alpha$.
\item
$\wedge$ denotes vector product of complex numbers: $a\wedge b:=a_1b_2-a_2b_1$. Similarly, $u\wedge \na v:=u_1\na v_2-u_2\na v_1$. 
\item
If $u:\Omega\to\C$ and if $\varpi$ is a $k$-form on $\Omega$ (with $k\in\llbracket 0, n-1\rrbracket$), then $\varpi\wedge (u\wedge\na u) $ denotes the $(k+1)$-form $\varpi\wedge(u_1d u_2-u_2du_1)$.
\item
We let $\R^n_+$ denote the open set $\R^{n-1}\times (0,\infty)$.
\een

%
\tableofcontents
\section{Crash course on Besov spaces}
\l{fun}
${}$

 We briefly recall here the basic properties of the  Besov spaces in $\R^n$, with special focus on the properties which will be instrumental for our purposes. For a complete treatment of these spaces, see \cite{triebel2,fjw,triebel3,runstsickel}. \par
\subsection{Preliminaries}
${}$

 In the sequel, ${\mathcal S}(\R^n)$ is the usual Schwartz space of rapidly decreasing $C^{\infty}$ functions. Let ${\mathcal Z}(\R^n)$ denote the subspace of
${\mathcal S}(\R^n)$ consisting of functions $\varphi\in {\mathcal S}(\R^n)$ such that $\partial^{\alpha}\varphi(0) = 0$ for every multi-index $\alpha\in \N^n$. Let ${\mathcal Z}^{\prime}(\R^n)$ stand for the topological dual of ${\mathcal Z}(\R^n)$. It is well-known \cite[Section 5.1.2]{triebel2} that ${\mathcal Z}^{\prime}(\R^n)$ can be identified with the quotient space ${\mathcal S}^{\prime}(\R^n)/{\mathcal P}(\R^n)$, where ${\mathcal P}(\R^n)$ denotes the space of all polynomials in $\R^n$. 

We denote by  ${\mathcal F}$ the Fourier transform.

 For all sequence $(f_j)_{j\geq 0}$ of measurable functions on $\R^n$, we set
\bes
\left\Vert (f_j)\right\Vert_{l^q(L^p)}:=\left(\sum_{j\geq 0} \left(\int_{\R^n} \left\vert f_j(x)\right\vert^pdx\right)^{q/p}\right)^{1/q},
\ees
with the usual modification when $p=\infty$ and/or $q=\infty$. If $(f_j)$ is labelled by $\Z$, then $\left\Vert (f_j)\right\Vert_{l^q(L^p)}$ is defined analogously with $\sum_{j\geq 0}$ replaced by $\sum_{j\in \Z}$.

 Finally, we fix some notation for finite order differences. Let $\Omega\subset \R^n$ be a domain and let $f:\Omega\to\R$. 
 For all integers $M\geq 0$, all $t>0$ and all $x, h\in \R^n$, set
\be
\l{ia1}
\Delta^M_hf(x)=\begin{cases}\d\sum_{l=0}^M {M \choose l} (-1)^{M-l} f(x+lh),&\text{if } x,\, x+h,\ldots,\,  x+Mh\in \Omega\\
0,&\text{otherwise}
\end{cases}.
\ee

\subsection{Definitions of Besov spaces}
\l{mm7}
${}$

 We first focus on inhomogeneous Besov spaces. Fix a sequence of functions $(\varphi_j)_{j\geq 0}\in {\mathcal S}(\R^n)$ such that:
\begin{itemize}
\item[$1.$] $\d \mbox{supp }\varphi_0\subset B(0,2)$ and $\d \mbox{supp }\varphi_j\subset B(0,2^{j+1})\setminus B(0,2^{j-1})  \mbox{ for all }j\geq 1$.
\item[$2.$]
For all multi-index $\alpha\in \N^n$, there exists $c_{\alpha}>0$ such that $\d
 \left\vert D^{\alpha}\varphi_j(x)\right\vert\leq c_{\alpha}2^{-j\left\vert \alpha\right\vert}$, for all $x\in \R^n$ and all $j\geq 0$.
\item[$3.$]
For all $x\in \R^n$, it holds $
\sum_{j\geq 0} \varphi_j(x)=1$.
\end{itemize}
\begin{definition} [Definition of inhomogeneous Besov spaces]
Let $s\in \R$, $1\leq p<\infty$ and $1\leq q\leq \infty$. Define $B^s_{p,q}(\R^n)$ as the space of tempered distributions $f\in {\mathcal S}^{\prime}(\R^n)$ such that
\bes
\left\Vert f\right\Vert_{B^s_{p,q}(\R^n)}:=\left\Vert \left(2^{sj} {\mathcal F}^{-1}\left(\varphi_j{\mathcal F}f(\cdot)\right)\right)\right\Vert_{l^q(L^p)}<\infty.
\ees
\end{definition}
Recall \cite[Section 2.3.2, Proposition 1, p. 46]{triebel2} that $B^s_{p,q}(\R^n)$ is a Banach space which does not depend on the choice of the sequence $(\varphi_j)_{j\geq 0}$, in the sense that two different choices for the sequence $(\varphi_j)_{j\geq 0}$ give rise to equivalent norms. Once the $\va_j$'s are fixed, we refer to the equality $f=\sum_j f_j$ in ${\mathcal S}'$ as the Littlewood-Paley decomposition of $f$.

 Let us now turn to the definition of homogeneous Besov spaces. Let $(\varphi_j)_{j\in \Z}$ be a sequence of functions satisfying:
\begin{itemize}
\item[$1.$]
$\d 
\mbox{supp }\varphi_j\subset B(0,2^{j+1})\setminus B(0,2^{j-1})  \mbox{ for all }j\in \Z$.
\item[$2.$]
For all multi-index $\alpha\in \N^n$, there exists $c_{\alpha}>0$ such that$\d 
\left\vert D^{\alpha}\varphi_j(x)\right\vert\leq c_{\alpha}2^{-j\left\vert \alpha\right\vert} $,  for all $x\in \R^n$ and all $j\in \Z$.
\item[$3.$]
For all  $x\in \R^n\setminus \left\{0\right\}$, it holds 
$
\sum_{j\in \Z} \varphi_j(x)=1$.
\end{itemize}
\begin{definition} [Definition of homogeneous Besov spaces]
Let $s\in \R$, $1\leq p<\infty$ and $1\leq q\leq \infty$. Define $\dot{B}^s_{p,q}(\R^n)$ as the space of $f\in {\mathcal Z}^{\prime}(\R^n)$ such that
\bes
\left\vert f\right\vert_{B^s_{p,q}(\R^n)}:=\left\Vert \left(2^{sj} {\mathcal F}^{-1}\left(\varphi_j{\mathcal F}f(\cdot)\right)\right)\right\Vert_{l^q(L^p)}<\infty.
\ees
\end{definition}
Note that this definition makes sense since, for all polynomial $P$ and all $f\in {\mathcal S}^{\prime}(\R^n)$, we have $\d
\left\vert f\right\vert_{B^s_{p,q}(\R^n)}=\left\vert f+P\right\vert_{B^s_{p,q}(\R^n)}$.

Again, $\dot{B}^s_{p,q}(\R^n)$ is a Banach space which does not depend on the choice of the sequence $(\varphi_j)_{j\in \Z}$ \cite[Section 5.1.5, Theorem,  p. 240]{triebel2}.\par
\noindent For all $s>0$ and all  $1\le p<\infty$,  $1\leq q\leq \infty$, we have  \cite[Section 2.3.3, Theorem]{triebel3}, \cite[Section 2.6.2, Proposition 3]{runstsickel}
\begin{equation} \label{homoglp}
B^s_{p,q}(\R^n)=L^p(\R^n)\cap \dot{B}^s_{p,q}(\R^n)\mbox{ and }\left\Vert f\right\Vert_{B^s_{p,q}(\R^n)}\sim \left\Vert f\right\Vert_{L^p(\R^n)}+\left\vert f\right\vert_{B^s_{p,q}(\R^n)}.
\end{equation}


 Besov spaces on domains of $\R^n$ are defined as follows.
\begin{definition} [Besov spaces on domains]
Let $\Omega\subset \R^n$ be an open set. Then
\begin{itemize}
\item[$1.$]
$\d 
B^s_{p,q}(\Omega):=\left\{f\in {\mathcal D}^{\prime}(\Omega);\mbox{ there exists }g\in B^s_{p,q}(\R^n)\mbox{ such that } f=g\vert_{\Omega}\right\}$,\\
equipped with the norm
\bes
\left\Vert f\right\Vert_{B^s_{p,q}(\Omega)}:=\inf \left\{\left\Vert g\right\Vert_{B^s_{p,q}(\R^n)};\, g\vert_{\Omega}=f\right\}.
\ees
\item[$2.$]
$\d 
\dot{B}^s_{p,q}(\Omega):=\left\{f\in {\mathcal D}^{\prime}(\Omega);\mbox{ there exists }g\in \dot{B}^s_{p,q}(\R^n)\mbox{ such that } f=g\vert_{\Omega}\right\}$,\\
equipped with the semi-norm
\bes
\left\Vert f\right\Vert_{\dot{B}^s_{p,q}(\Omega)}:=\inf \left\{\left\Vert g\right\Vert_{\dot{B}^s_{p,q}(\R^n)};\, g\vert_{\Omega}=f\right\}.
\ees
\end{itemize}
\end{definition}
Local Besov spaces are defined in the usual way: $f\in B^s_{p,q}$ near a point $x$ if for some cutoff $\va$ which equals $1$ near $x$ we have $\va f\in B^s_{p,q}$.
If $f$ belongs to $B^s_{p,q}$ near each point, then we write $f\in (B^s_{p,q})_{loc}$.

The following is straightforward. 
\bl
\l{ka3}
Let $f:\Omega\to\R$. If, for each $x\in\overline\Omega$, $f\in B^s_{p,q}(B(x,r)\cap \Omega)$ for some $r=r(x)>0$, then $f\in B^s_{p,q}$.
\el


 \subsection{Besov spaces on $\T^n$}
 \l{mm6}
 ${}$
 
Let $\varphi_0\in {\mathcal D}(\R^n)$ be such that
\bes
\varphi_0(x)=1\mbox{ for all } \left\vert x\right\vert<1\mbox{ and }\varphi_0(x)=0\mbox{ for all }\left\vert x\right\vert\geq \frac 32.
\ees
For all $k\geq 1$ and all $x\in \R^n$, define
\bes
\varphi_k(x):=\varphi_0(2^{-k}x)-\varphi_0(2^{-k+1}x).
\ees 
\begin{definition} \label{periodicbesov}
Let $s\in \R$, $1\leq p<\infty$ and $1\le q\leq \infty$. Define $B^s_{p,q}(\T^n)$ as the space of distributions 
$f\in {\mathcal D}^{\prime}(\T^n)$ whose Fourier coefficients $(a_m)_{m\in\Z^n}$ satisfy
\bes
\left\Vert f\right\Vert_{B^s_{p,q}(\T^n)}:=\left(\sum_{j=0}^{\infty} 2^{jsq} \left\Vert x\mapsto \sum_{m\in \Z^n} a_m\varphi_j(2\pi m)e^{2\im\pi m\cdot x}\right\Vert_{L^p(\T^n)}^q\right)^{1/q}<\infty
\ees
(with the usual modification when $q=\infty$). Again, the choice of the system $(\va_j)_{j\ge 0}$ is irrelevant, and the equality $f=\sum f_j$, with $f_j:=\sum_{m} a_m\varphi_j(2\pi m)e^{2\im\pi m\cdot x}$, is the Littlewood-Paley decomposition of $f$.
\end{definition}
Alternatively, we have  $f\in B^s_{p,q}(\T^n)$ if and only if $f$ can be identified with a $\Z^n$-periodic distribution in $\R^n$, still denoted $f$, which belongs to $(B^s_{p,q})_{loc}(\R^n)$ \cite[Section 3.5.4, pp. 167-169]{schmeisser}.

\subsection{Characterization by differences}
\l{mm5}
${}$

Among the various characterizations of Besov spaces, we recall here the ones involving differences \cite[Section 5.2.3]{triebel2}, \cite[Theorem, p. 41]{runstsickel}, \cite[Section 1.11.9,  Theorem 1.118, p. 74]{triebel06}.
 \begin{proposition}
\l{p2.4} 

Let $s>0$, $1\leq p<\infty$ and $1\leq q\leq \infty$. Let $M>s$ be an integer. Then, with the usual modification when $q=\infty$:
\begin{itemize}
\item[$1.$] In the space $\dot B^s_{p,q}(\R^n)$ we have the equivalence of semi-norms
\begin{equation} \label{equivnormhomogrn}
\begin{aligned}
\left\vert f\right\vert_{B^s_{p,q}(\R^n)}\sim &\left(\int_{\R^n} \left\vert h\right\vert^{-sq}
\left\|\Delta_h^Mf\right\|_{L^p (\R^n)}^q\, 
\frac{dh}{\left\vert h\right\vert^n}\right)^{1/q}\\
\sim & \sum_{j=1}^n\left(\int_{\R} \left\vert h\right\vert^{-sq}
\left\|\Delta_{h e_j}^Mf\right\|_{L^p (\R^n)}^q\, 
\frac{dh}{\left\vert h\right\vert}\right)^{1/q}.
\end{aligned}
\end{equation}
\item[$2.$]
The full $B^s_{p,q}$ norm satisfies, for all $\delta>0$,
\bes
\left\Vert f\right\Vert_{B^s_{p,q}(\R^n)}\sim \left\Vert f\right\Vert_{L^p(\R^n)}+\left(\int_{\left\vert h\right\vert\le \delta} \left\vert h\right\vert^{-sq}
\left\|\Delta_h^Mf\right\|_{L^p (\R^n)}^q\, 
\frac{dh}{\left\vert h\right\vert^n}\right)^{1/q}. 
\ees
\end{itemize}
\end{proposition}
\subsection{Characterization by harmonic extensions}
\l{chha}
${}$

In Section \ref{pos}, it will be convenient to work with extensions of maps in $B^{s}_{p,q}$. The connection between regularity of maps and the properties of their suitable extensions is a classical topic in the theory of function spaces. Here is a typical result in this direction. 
It characterizes $B^{s}_{p,q}$ by means of the harmonic extension \cite{triebelheat}, \cite[Section 2.12.2, Theorem, p. 184]{triebel2}. More specifically, if $f$ is measurable in $\R^n$  and $s\in (0,1)$,  then we have
\begin{equation} \label{besovnorm}
\left\Vert f\right\Vert_{B^s_{p,q}(\R^n)}\sim \left\Vert f\right\Vert_{L^p(\R^n)}+ \left(\int_0^{\infty} t^{(1-s)q}\left\Vert \frac{\partial P_tf}{\partial t}(\cdot)\right\Vert_{L^p(\R^n)}^q \frac{dt}t\right)^{1/q},
\end{equation}
where $P_t$ stands for the Poisson semigroup generated by $-\Delta$, so that $(x,t)\mapsto P_tf(x)$, $t>0$, $x\in\R^n$,  is the harmonic extension of $f$ to the upper-half space. 
Since when $p>1$ we have
\bes\left\Vert \frac{\partial P_tf}{\partial t}\right\Vert_{L^p(\R^n)} =\left\Vert (-\Delta_x)^{1/2}P_tf\right\Vert_{L^p(\R^n)}\sim \left\Vert \nabla_x P_tf\right\Vert_{L^p(\R^n)},
\ees
one also has, for $1<p<\infty$ and $1\le q\le \infty$, 
\begin{equation} \label{besovnormbis}
\left\Vert f\right\Vert_{B^s_{p,q}(\R^n)}\sim \left\Vert f\right\Vert_{L^p(\R^n)}+ \left(\int_0^{\infty} t^{(1-s)q}\left\Vert \nabla P_tf(\cdot)\right\Vert_{L^p(\R^n)}^q \frac{dt}t\right)^{1/q}
\end{equation}
(with the usual modification when $q=\infty$).

The results in the literature are not suited to our context. We will need  some variants of \eqref{besovnormbis}, which will be stated and proved in Section 
\ref{characext} below.

\subsection{Lizorkin type characterizations}
\l{mm1}
${}$

Such characterizations involve restrictions of the Fourier transform on cubes or corridors; see e.g. \cite[Section 2.5.4, pp. 85-86]{triebel2} or \cite[Section 3.5.3, pp. 166-167]{schmeisser}. The following special case \cite[Section 3.5.3, Theorem, p. 167]{schmeisser} will be useful later. 
\bpr
\l{mm2}
Let $s\in\R$, $1<p<\infty$ and $1\le q\le\infty$. Set $K_0:=\{0\}\subset\Z^n$ and, for $j\ge 1$, let $K_j:=\{ m\in\Z^n;\, 2^{j-1}\le |m|<2^j\}$.\footnote{ Here, $|m|:=\max_{l=1}^n|m_l|$.} Let $f\in{\cal D}'(\T^n)$ have the Fourier series expansion  $f=\sum_{m\in\Z^n} a_me^{2\im\pi m\cdot x}$. We set   $f_j:=\sum_{m\in K_j} a_me^{2\im\pi m\cdot x}$. Then we have the norm equivalence
\bes
\|f\|_{B^s_{p,q}(\T^n)}\sim \left(\sum_{j=0}^{\infty} 2^{jsq} \left\Vert f_j\right\Vert_{L^p(\T^n)}^q\right)^{1/q}
\ees
(with the usual modification when $q=\infty$). 
\epr

\subsection{Characterization by the Haar system}
\l{at7}
${}$

Besov spaces can also be described via the size of their wavelet coefficients. To illustrate this, we start with  low smoothness Besov spaces, which can be described using the Haar basis. (The next section is devoted to smoother spaces and bases.) For the results of this section, see e.g. \cite[Corollary 5.3]{devorepopov}, \cite[Section 7]{bourdaud}, \cite[Theorem 1.58]{triebel06},  \cite[Theorem 2.21]{triebel10}. \par
\noindent Let
\be
\l{qa8}
\psi_M(x):=
\begin{cases}
1, & \text{if }0\leq x<1/2\\
-1, & \text{if }1/2\leq x\leq 1\\
0, & \text{if }x\notin [0,1]
\end{cases}, \text{ and }\psi_F(x):=\left\vert \psi_M(x)\right\vert.
\ee

When  $j\in\N$, we let 
\be
\l{qa1}
G^j:=\begin{cases}
\left\{F,M\right\}^{n},&\text{if }j=0\\
\left\{F,M\right\}^{n}\setminus\{ (F, F, \ldots, F)\},&\text{if }j>0
\end{cases}.
\ee
For all $m\in \Z^n$, all $x\in \R^n$ and all $G\in \left\{F,M\right\}^{n}$, define
\be
\l{qa2}
\Psi_m^G(x):=\prod_{r=1}^n \psi_{G_r}(x_r-m_r).
\ee
%
Finally, for all $m\in \Z^n$, all $j\in \N$, all $G\in 
G^j$
and all $x\in \R^n$, let
\be
\l{qa3}
\Psi_m^{j,G}(x):=
\begin{cases}
\displaystyle \Psi_m^G(x), & \text{if }j=0\\
\displaystyle 2^{nj/2}\Psi^G_m(2^jx), & \text{if }j\geq 1
\end{cases}.
\ee
Recall  that the family $(\Psi_m^{j,G})$, called the Haar system, is an orthonormal basis of $L^2(\R^n)$ \cite[Proposition 1.53]{triebel06}. Moreover, we have the following result \cite[Theorem 2.21]{triebel10}.
\bpr
\l{at11}
Let $s>0$, $1\leq p<\infty$,  and $1\le q\le \infty$ be such that $sp<1$. Let $f\in {\mathcal S}^{\prime}(\R^n$). Then $f\in B^s_{p,q}(\R^n)$ if and only if there exists a sequence $\left(\mu^{j, G}_{m}\right)_{j\geq 0,\ G\in G^j,\ m\in \Z^n}$ such that
\be
\l{qa4}
\sum_{j=0}^{\infty}\  \sum_{G\in G^j} \left(\sum_{m\in \Z^n} \left\vert \mu^{j, G}_{m}\right\vert^p\right)^{q/p}<\infty
\ee
(obvious modification when $q=\infty$) and
\be \label{decompof}
f=\sum_{j=0}^{\infty}\  \sum_{G\in G^j} \sum_{m\in \Z^n} \mu^{j, G}_{m} 2^{-j\left(s-n/p\right)}  2^{-nj/2}  \Psi^{j,G}_{m}.
\ee

Here, the series in \eqref{decompof} converges unconditionally in $B^s_{p,q}(\R^n)$ when $q<\infty$. Moreover,
\be
\l{qa5}
\left\Vert f\right\Vert_{B^s_{p,q}(\R^n)}\sim \left(\sum_{j=0}^{\infty}\ \sum_{G\in G^j} \left(\sum_{m\in \Z^n} \left\vert \mu^{j, G}_{m}\right\vert^p\right)^{q/p}\right)^{1/q}
\ee
(obvious modification when $q=\infty$).
\epr
Equivalently, Proposition \ref{at11} can be reformulated as follows. Consider the partition of $\R^n$ into standard dyadic cubes   $Q$  of side $2^{-j}$. \footnote{ Thus the  $Q$'s  are of the form  $Q=2^{-j}\prod_{k=1}^n[m_k,m_k+1)$, with $m_k\in\Z$.} For all  $x\in\R^n$, denote by $Q_j(x)$  the unique dyadic cube of side $2^{-j}$ containing $x$. If $f\in L^1_{loc}(\R^n)$, define $E_j(f)(x):=\fint_{Q_j(x)}f$  for all $j\geq 0$. We also set $E_{-1}(f):=0$.
We have the following results (see \cite[Theorem 5 with $m=0$]{bourdaud} in $\R^n$; see also \cite[Appendix A]{lss} in the framework of Sobolev spaces  on $\T^n$).
\bpr\label{caracBesov} Let $s>0$, $1\leq p<\infty$,  and $1\le q\le \infty$ be such that $sp<1$. Let $f\in L^1_{loc}(\R^n$). 
Then 
\bes\label{caracBesovBourdaud}
\left\Vert f\right\Vert^q_{B^s_{p,q}(\R^n)}\sim \sum_{j \ge 0} 2^{sjq}\|E_j(f)-E_{j-1}(f)\|_{L^p}^q
\ees
(obvious modification when $q=\infty$).
\epr

Similar results hold when $\R^n$ is replaced by $(0,1)^n$ or $\T^n$;  it suffices to consider only dyadic cubes contained in $[0,1)^n$.

\bc
\l{mq2}
Let $s>0$, $1\leq p<\infty$,  and $1\le q\le \infty$ be such that $sp<1$. Let $f\in L^1_{loc}(\R^n$). 
Then 
\bes\label{mq3}
\left\Vert f\right\Vert^q_{B^s_{p,q}(\R^n)}\sim \sum_{j \ge 0} 2^{sjq}\|f-E_{j}(f)\|_{L^p}^q
\ees
(obvious modification when $q=\infty$).

Similar results hold when $\R^n$ is replaced by $(0,1)^n$ or $\T^n$.
\ec
\bc 
\l{mp1}
Let $s>0$, $1\leq p<\infty$,  and $1\le q\le \infty$ be such that $sp<1$. Let $(\varphi_j)_{j\ge 0}$ be a sequence of functions on $(0,1)^n$ such that: for any $j$,   $\varphi_j$ is constant on each dyadic cube  $Q$ of size $2^{-j}$. 
Assume that  
$
\sum_{j \ge 1} 2^{sjq}\|\varphi_j-\varphi_{j-1}\|_{L^p}^q < \infty$. 
Then $(\varphi_j)$ converges in $L^p$ to some    $\varphi \in B^s_{p,q}$, and we have
\bes
\left\Vert \varphi \right\Vert_{B^s_{p,q}((0,1)^n)} \lesssim  \left(\sum_{j \ge 0} 2^{sjq}\|\varphi_j-\varphi_{j-1}\|_{L^p}^q\right)^{1/q}
\ees
(with the convention $\va_{-1}:=0$ and with the usual modification when $q=\infty$).
\ec
In the framework of Sobolev spaces, Corollaries \ref{mq2} and \ref{mp1} are easy consequences of Proposition \ref{caracBesov}; see \cite[Appendix A, Theorem A.1]{lss} and \cite[Appendix A, Corollary A.1]{lss}. The arguments in \cite{lss} apply with no changes to Besov spaces. Details are left to the reader.

\subsection{Characterization via smooth wavelets}
\l{qa6}

 Proposition \ref{at11} has a counterpart when $sp\ge 1$; this requires smoother \enquote{mother wavelet} $\psi_M$ and \enquote{father wavelet} $\psi_F$.  Given $\psi_F$ and $\psi_M$ two real functions, define $\psi_m^{j, G}$ as in \eqref{qa1}--\eqref{qa3}. Then \cite[Chapter 6]{meyer92}, \cite[Section 1.7.3]{triebel06} for every integer $k>0$ we may find some $\psi_F\in C^k_c(\R)$ and $\psi_M\in C^k_c(\R)$ such that the following result holds.
\bpr
\l{qb1}
Let $s>0$, $1\leq p<\infty$,  and $1\le q\le \infty$ be such that $s<k$. Let $f\in {\mathcal S}^{\prime}(\R^n$). Then $f\in B^s_{p,q}(\R^n)$ if and only if there exists a sequence $\left(\mu^{j, G}_{m}\right)_{j\geq 0,\ G\in G^j,\ m\in \Z^n}$ such that
\be
\l{qb2}
\sum_{j=0}^{\infty}\  \sum_{G\in G^j} \left(\sum_{m\in \Z^n} \left\vert \mu^{j, G}_{m}\right\vert^p\right)^{q/p}<\infty
\ee
(obvious modification when $q=\infty$) and
\be \label{qb3}
f=\sum_{j=0}^{\infty}\  \sum_{G\in G^j} \sum_{m\in \Z^n} \mu^{j, G}_{m} 2^{-j\left(s-n/p\right)}  2^{-nj/2}  \Psi^{j,G}_{m}.
\ee

Here, 
the series in \eqref{decompof} converges unconditionally in $B^s_{p,q}(\R^n)$ when $q<\infty$. Moreover,
\be
\l{qb4}
\left\Vert f\right\Vert_{B^s_{p,q}(\R^n)}\sim \left(\sum_{j=0}^{\infty}\ \sum_{G\in G^j} \left(\sum_{m\in \Z^n} \left\vert \mu^{j, G}_{m}\right\vert^p\right)^{q/p}\right)^{1/q}
\ee
(obvious modification when $q=\infty$).
\epr

For further use, let us note that, if $f\in B^s_{p,q}(\R^n)$ for some $s>0$, $1\le p<\infty$ and $1\le q\le\infty$,  then we have
\be
\l{qb40}
\mu^{j, G}_{m}=\mu^{j, G}_{m}(f)=2^{j(s-n/p+n/2)}\, \int_{\R^n} f(x)\, \Psi^{j,G}_{m}(x)\, dx.
\ee

This immediately leads to the following consequence of Proposition \ref{qb1}, the proof of which is left to the reader. 
\bc
\l{qb400}
Let $s>0$, $1\le p<\infty$ and $1\le q\le\infty$ be such that $s<k$. Assume that $f\in L^p(\R^n)$ is such that the coefficients $\mu^{j, G}_{m}$ given by \eqref{qb40} satisfy 
\be
\l{qb50}
\sum_{j=0}^{\infty}\  \sum_{G\in G^j} \left(\sum_{m\in \Z^n} \left\vert \mu^{j, G}_{m}\right\vert^p\right)^{q/p}=\infty
\ee
(obvious modification when $q=\infty$). Then $f\not\in B^s_{p,q}(\R^n)$.
\ec

\subsection{Nikolski\u\i{} type decompositions}
\l{mm8}
${}$

In practice, we often do not know the Littlewood-Paley decomposition of some given $f$, but only a Nikolski\u\i{} representation (or decomposition) of $f$. More specifically, set $\mathcal{C}_j:=B(0,2^{j+2})$, with $j\in\N$. Let $f^j\in\mathcal{S}'$ satisfy
\be\l{e230501}
\supp{\mathcal F}f^j\subset\mathcal{C}_j,\ \fo j\in\N,\ \text{and } f=\sum_jf^j\text{ in }{\mathcal S}';
\ee
the decomposition $f=\sum_j f^j$ is a Nikolski\u\i{} decomposition of $f$. Note that the Littlewood-Paley decomposition is a special Nikolski\u\i{} decomposition.
%

We have the following result.
\bpr
\l{mm9}
Let $s>0$, $1\le p<\infty$, $1\le q\le\infty$. Assume that \eqref{e230501} holds. Then  we have
\be
\l{28024}
\Big\|\sum_{j}f^j\Big\|_{B^{s}_{p,q}}\lesssim\left(\sum_{j}2^{ sqj }\|f^j\|^q_{L^p}\right)^{1/q},
\ee
with the usual modification when $q=\infty$.
\epr
\noindent
The above was proved in \cite[Lemma 1]{gnp} (see also \cite{yamazaki}) in the framework of Triebel-Lizorkin spaces $F^s_{p,q}$; the proof applies with no change to Besov spaces and will be omitted here. For related results in the framework of Besov spaces, see \cite[Section 2.5.2, pp. 79-80]{triebel2} and \cite[Section 2.3.2, Theorem, p. 105]{schmeisser}.

\section{Positive cases}
\l{pos}
${}$

We start with the trivial case.
\begin{case}
\l{tri}
{\it Range.} $s>0$, $1\le p<\infty$, $1\le q\le\infty$, and $sp>n$.

\smallskip
\noindent
{\it Conclusion.}
$B^s_{p,q}(\Omega ; \so)$ does  have the lifting property.
\end{case}
\bp
Since $B^s_{p,q}(\Omega)\ho C^0(\overline\Omega)$ (Lemma \ref{ka1}), we may  write $u=e^{\im\va}$, with $\va$ continuous. Locally, we have $\va=-\im \ln u$, for some smooth determination $\ln$ of the complex logarithm. Then $\va$ belongs to $B^s_{p,q}$ locally in $\overline\Omega$ (Lemma \ref{ka2}), and thus globally (Lemma \ref{ka3}). 
\ep

\begin{case}
\l{A}
{\it Range.} $0<s<1$, $1\le p<\infty$, $1\le q\le\infty$, and $sp<1$.

\smallskip
\noindent
{\it Conclusion.}
$B^s_{p,q}(\Omega ; \so)$ does have the lifting property.
\end{case}
\bp The argument being essentially the one in  \cite[Section 1]{lss}, we will be sketchy. Assume for simplicity that $\Omega=(0,1)^n$. 
Let $u \in B^s_{p,q}(\Omega ; \so)$. For all $j\in\N$,   consider the function $U_j$ defined by 
\bes
U_j(x):=
\begin{cases}
\d E_j(u)(x)/|E_j(u)(x)|,&\mbox{if }E_j(u)(x) \neq 0\\
1,&\mbox{if }E_j(u)(x) = 0
\end{cases}.
\ees
Since $E_j(u)\to u$ a.e., we find that  $U_j \rightarrow u$ a.e. on $\Omega$. By induction on $j$, for all $j\in\N$ 
we construct a phase $\va_j$ of $U_j$, constant on each dyadic cube of size $2^{-j}$, and satisfying  the inequality
\be\l{mq1}
|\varphi_j -\varphi_{j-1}| \leq \pi |U_j-U_{j-1}| \quad\mbox{on }\Omega,\, \fo j\ge 1.\footnotemark
\ee
\footnotetext{ Thus $\va_j$ is the phase of $U_j$ closest to $\va_{j-1}$.}
As in \cite{lss}, \eqref{mq1} implies 
\bes
|\varphi_j -\varphi_{j-1}| \lesssim |u -E_j(u)|+ |u -E_{j-1}(u)|,
\ees
and thus, e.g. when $q<\infty$, we have
\bes
\sum_{j \ge 1} 2^{sjq}\|\varphi_j-\varphi_{j-1}\|_{L^p}^q \lesssim \sum_{j \ge 0} 2^{sjq}\|u-E_{j}(u)\|_{L^p}^q.
\ees
Applying Corollaries \ref{mq2} and \ref{mp1}, 
 we obtain that $\varphi_j \rightarrow \varphi $ in $L^p$ to some $\varphi \in B^s_{p,q}(\Omega ; \R)$. Since $\va_j$ is a phase of $U_j$ and $U_j\to u$ a.e., we find that  $\va$ is a phase of $u$. In addition,    we have the control 
$\|\varphi\|_{B^s_{p,q}} \lesssim \|u\|_{B^s_{p,q}}$.
\ep

\begin{case}
\l{X}
{\it Range.} $0<s<1$, $1\le p<\infty$, $1\le q<\infty$, and $sp=n$.

\smallskip
\noindent
{\it Conclusion.}
$B^s_{p,q}(\Omega ; \so)$ does  have the lifting property.
\end{case}
\begin{proof} Here, it will be convenient to work with $\Omega=\T^n$. 
Let $|\ |$ denote the sup norm in $\R^n$. Let $\rho\in C^\infty$ be a mollifier supported in $\{ |x|\le 1\}$ and set $F(x,\ve):=u\ast\rho_\ve(x)$, $x\in\T^n$, $\ve>0$. Since $sp=n$, we have  $u\in \VMO(\T^n)$, by Lemma \ref{B-VMO}. Let us recall that, if $u\in\VMO(\T^n ; \so)$ then, for  some $\delta>0$ (depending on $u$) we have \cite[Remark 3, p. 207]{brezisnirenberg1}
\begin{equation} \label{boundsv}
\frac 12<\left\vert F(x,\ve)\right\vert\le 1\mbox{ for all } x\in \T^n\mbox{ and all }\ve\in (0,\delta).\footnotemark
\end{equation}
\footnotetext{ For an explicit calculation leading to \eqref{boundsv}, see e.g. \cite[p. 415]{surveypetru}.}
Define
\bes
w(x,\ve):=\frac{F(x,\ve)}{\left\vert F(x,\ve)\right\vert}\mbox{ for all }x\in \T^n\mbox{ and all }\ve\in (0,\delta).
\ees
Pick up a function $\psi\in C^{\infty}(\T^n\times (0,\delta) ; \R)$ such that $w=e^{\im \psi}$.
We note that for all $j\in\llbracket 1, n\rrbracket$ we have
$
\nabla\psi=-\im  \overline{w}\nabla w
$, and  
$
\partial_j |F|= |F|^{-1}(F\partial_j\overline{F}+\overline{F}\partial_jF)/2
$. Therefore, 
\eqref{boundsv} yields
\begin{equation} \label{nablaw}
\left\vert \nabla\psi\right\vert=\left\vert \nabla w\right\vert\lesssim \left\vert \nabla F\right\vert.
\end{equation}
%
%
%
%
%
In view of \eqref{nablaw} and estimate \eqref{cg1} in Lemma \ref{ab1}, we find that
\be
\l{ka5}
 |u|_{B^{s,p}_q(\T^n)}^q  \gtrsim  \displaystyle \int_0^\delta\ve^{q-sq}\|(\na F)(\cdot,\ve)\|_{L^p}^q\, \frac{d\ve}{\ve} \gtrsim  \int_0^\delta\ve^{q-sq}\|(\na \psi)(\cdot,\ve)\|_{L^p}^q\, \frac{d\ve}{\ve}.
\ee
Combining \eqref{ka5} with  the conclusion of Lemma \ref{ab1}, we obtain that the phase $\psi$ has, on $\T^n$,  a trace $\va\in B^s_{p,q}$, in the sense that the limit $\va:=\lim_{\ve\to 0}\psi(\cdot,\ve)$ exists in $B^s_{p,q}$. In particular (using Lemma \ref{kc2}), we have that $\psi(\cdot, \ve_j)\to\va$ a.e. along some sequence $\ve_j\to 0$; this leads to $w(\cdot, \ve_j)=e^{\im\psi(\cdot,\ve_j)}\to e^{\im\va}$ a.e. Since, on the other hand, we have $\lim_{\ve\to 0}w(\cdot,\ve)=u$ a.e., we find that $\va$ is a $B^s_{p,q}$ phase of $u$.
%
\end{proof}\Bk

The next case is somewhat similar to Case \ref{X}, so that our argument is less detailed.
\begin{case}
\l{kc3}
{\it Range.} $s=1$, $p=n$, $1\le q<\infty$.

\smallskip
\noindent
{\it Conclusion.}
$B^1_{n,q}(\Omega ; \so)$ does  have the lifting property.
\end{case}
\bp
We consider  $\delta$, $w$ and $\psi$ as in Case \ref{X}. The analog of \eqref{nablaw} is the estimate 
\be
\l{kc4}
|\p_j\p_k\psi|+|\na\psi|^2\lesssim |\p_j\p_kF|+|\na F|^2,
\ee
which is a straightforward consequence of the identities
\bes
\na\psi=-\im \overline w\na w\text{ and }\p_j\p_k\psi=-\im \overline w\p_j\p_kw+\im w^2\p_j w\p_kw.
\ees
Combining \eqref{kc4} with the second  part of Lemma \ref{kb2}, we obtain  
\be
\l{kg1a}
|u|_{B^1_{n,q}}^q\gtrsim \int_0^\delta\ve^q\left(\sum_{j, k=1}^n\left\|\p_j\p_k\psi(\cdot,\ve)\right\|_{L^n}^q+\left\|\p_\ve\p_\ve\psi(\cdot,\ve)\right\|_{L^n}^q+\|\na\psi(\cdot,\ve)\|_{L^{2n}}^{2q}\right)\, \frac{d\ve}{\ve}.
\ee
By \eqref{kg1a} and the first part of Lemma \ref{kb2}, we find that  $\psi$ has a trace $\va:=\tr \psi \in B^1_{n,q}(\T^n)$. Clearly, $\va$ is a $B^1_{n,q}$  phase of $u$.
\ep

\begin{case}
\l{Y}
{\it Range.} $s>1$, $1\le p<\infty$, $1\le q<\infty$, $n=2$, and $sp=2$.

Or
$s>1$, $1\le p<\infty$, $1\le q\le p$, $n\ge 3$, and $sp= 2$. 

Or: $s>1$, $1\le p<\infty$, $1\le q\le \infty$, $n\ge 2$,
and $sp> 2$.

\smallskip
\noindent
{\it Conclusion.}
$B^s_{p,q}(\Omega ; \so)$ does  have the lifting property.
\end{case}
Note that, in the critical case where $sp=2$, our result is weaker in dimension $n\ge 3$ (when we ask $1\le q\le p$) than in dimension $2$ (when we merely ask $1\le q<\infty$).
\bp
The general strategy is the same as in \cite[Section 3, Proof of Theorem 3]{lss},\footnote{ See also \cite{carbou}.} but the key argument (validity of \eqref{at1} below) is much more involved in our case. 

It will be convenient to work in $\Omega=\T^n$. Let $u\in B^s_{p, q}(\T^n ; \so)$. Assume first that we do may write $u=e^{\im\va}$, with $\va\in B^s_{p, q}((0,1)^n ; \R)$. Then $u, \va\in W^{1,p}$ (Lemma \ref{kc2}). We are thus in position to apply chain's rule and infer that $\na u=\im u\na \va$, and therefore
\be
\l{at2}
\na\va=\frac 1{\im u}\na u=F,\ \text{with }F:=u\wedge\na u\in L^p(\T^n ; \R^n).
\ee
The assumptions on $s$, $p$, $q$ imply that $F\in B^{s-1}_{p, q}$ (Lemma \ref{at3}). We may now argue as follows. If $\va$ solves \eqref{at2}, then $\na\va\in B^{s-1}_{p, q}$, and thus $\va\in B^s_{p, q}$ (Lemma \ref{at4}). Next, since  $u, \va\in W^{1,p}\cap L^{\infty}$, we find that 
\bes
\na(u\, e^{-\im\va})=\na u\, e^{-\im\va}-\im u\, e^{-\im\va}\na\va=\im u\, e^{-\im\va}(u\wedge\na u-\na\va)=0.
\ees
Thus $u\, e^{-\im\va}$ is constant, and therefore  $\va$ is, up to an appropriate additive constant, a $B^s_{p, q}$ phase of $u$. 

There is a flaw in the above. Indeed, \eqref{at2} need not have a solution. In $\T^n$, the necessary and sufficient conditions for the solvability of \eqref{at2} are\footnote{ This is easily seen by an inspection of the Fourier coefficients.}
\be
\l{at5}
\int_{\T^n}F=\widehat F(0)=0
\ee
and
\be
\l{at1}
\curl F=0.
\ee
Clearly, \eqref{at5} holds.\footnote{ Expand $u\wedge\na u$ in Fourier series.} We complete Case \ref{Y} by noting that \eqref{at1} holds in the relevant range of $s$, $p$, $q$ and $n$ (Lemma \ref{at6}). 
\ep

\begin{remark}
\l{aa1} We briefly discuss the lifting problem when $s\le 0$. For such $s$, distributions in $B^s_{p,q}$ need not be integrable functions, and thus the meaning of the equality $u=e^{\im\va}$ is unclear. We therefore address the following reasonable version of the lifting problem: let $u:\Omega\to \so$ be a measurable function such that $u\in  B^s_{p, q}(\Omega)$.  Is there any $\va\in L^1_{loc}\cap B^s_{p,q}(\Omega ; \R)$ such that $u=e^{\im\va}$? 

Let us note that the answer is trivially positive when $s<0$, $1\le p<\infty$, $1\le q\le\infty$.

Indeed, let $\va$ be any bounded measurable lifting of $u$. Then $\va\in B^s_{p, q}$, since  $L^\infty\ho B^s_{p,q}$ when $s<0$ (see Lemma \ref{ia2}). 
\end{remark}

\section{Negative cases}
\l{neg}
 \begin{case}
\l{B}
{\it Range.} $0<s<1$, $1\le p< \infty$, $1\le q<\infty$, $n\ge 2$, and $1\leq sp <n$.

Or $0<s<1$, $1\le p< \infty$, $q=\infty$, $n\ge 2$, and $1< sp <n$.

\smallskip
\noindent
{\it Conclusion.} $B^s_{p,q}(\Omega ; \so)$ does not have the lifting property.
\end{case}

%
%
     
 \begin{proof}
We want to show  that there exists a function $u\in B^{s}_{p,q}$ such that $u\neq e^{{\im} \va}$  for any $\va \in B^{s}_{p,q}$.
 
 For sufficiently small $\ve>0$, set $
 s_1:=s/(1-\ve)$ and $p_1:=(1-\varepsilon)p$. By Lemma \ref{Besovemb}, we have $B^{s_1}_{p_1,q_1}\not\hookrightarrow B^s_{p,q}$ (for any $q_1$). We will use later this fact for $q_1:=(1-\varepsilon)q$.

Let $\psi\in B^{s_1}_{p_1,q_1}\setminus B^s_{p,q}$ and set $u:=e^{{\im} \psi}$. Then $u\in B^{s_1}_{p_1,q_1}\cap L^{\infty}$ (Lemma \ref{eipsi}) and thus $u\in B^{s}_{p,q}$
(Lemma \ref{gn}). 

 We claim that there is no $\varphi\in B^s_{p,q}$ such that $u=e^{{\im} \varphi}$. Argue by contradiction. Since $u=e^{\im\va}=e^{\im\psi}$, the function $(\varphi-\psi)/2\pi$ belongs to $(B^s_{p,q}+B^{s_1}_{p_1,q_1})(\Omega ; \Z)$. By Lemma \ref{Eunicite}, this implies that $\varphi-\psi$ is constant, and thus $\psi\in B^{s}_{p,q}$, which is a contradiction. 
\end{proof}
 \begin{case}
\l{xa2}
{\it Range.}
$0<s<\infty$, $1\le p<\infty$, $1\le q< \infty$, $n\ge 2$, and $1\le sp<2$. 

Or $0<s<\infty$, $1\le p<\infty$, $ q= \infty$, $n\ge 2$, and $1<sp\le 2$.

\smallskip
\noindent
{\it Conclusion.}
$B^s_{p,q}(\Omega ; \so)$ does not have the lifting property.
\end{case}
\bp
The proof is based on the example of a topological obstruction considering the case $n=2$. Consider the map $\d 
u(x)=\frac{x}{|x|}$, $\forall\,  x\in\R^2$.

 We first prove that $u \in B^s_{p,q}(\Omega)$ for any smooth bounded domain $\Omega\subset\R^2$.  We  distinguish  two cases: firstly, $q \le \infty$ and $sp <2$ and secondly, $q=\infty$ and $sp=2$. 

In the first case,  let $s_1>s$ such that $s_1$ is not an integer and $1<s_1p<2$, which implies $W^{s_1,p}=B^{s_1}_{p,p} \ho B^{s}_{p,q}$. Since
$u \in W^{s_1,p}$ \cite[Section 4]{lss}, we find that $u\in B^{s}_{p,q}$. 


 The second case is slightly more  involved.  By the Gagliardo-Nirenberg inequality (Lemma \ref{gn} below), it suffices to prove that $u \in B^2_{1,\infty}(\Omega)$. Using  Proposition \ref{p2.4}, a sufficient condition for this to hold is 
\be
\l{qf1}
\left\Vert \Delta^3_{h}u\right\Vert_{L^1(\R^2)}\lesssim \left\vert h\right\vert^2,\ \fo h\in \R^2.
\ee

Since $u$ is radially symmetric and $0$-homogeneous, this amounts to checking \Bk that
\be\l{delta3l1}
\|\Delta_{e_1}^3u\|_{L^1(\mathbb R^2)}<\infty. 
\ee

 However, by the mean-value theorem, for all $\left\vert x\right\vert\ge 1$ we have
\be\l{delta3infty}
|\Delta^3_{e_1} u(x)|\lesssim 1/|x|^3,
\ee
while $\Delta^3_{e_1}u$ is bounded in $B(0,1)$  since $u$ is $\so$-valued.  Using this fact and estimate \eqref{delta3infty}, we obtain \eqref{delta3l1}.

We next claim that $u$ has no $B^s_{p,q}$ lifting in $\Omega$ provided $\Omega\subset\R^2$ is a smooth bounded domain containing the origin. Argue by contradiction, and 
 assume  that $u=e^{\im\va}$ for some $\va\in B^s_{p,q}(\Omega)$. Let, as in \cite[p. 50]{lss},  $
 \theta\in C^{\infty}(\R^2\setminus ([0,\infty)\times \left\{0\right\}))$ be such that $e^{\im\theta}=u$.
 
  Note that $\theta\in B^s_{p,q}(\omega)$ for every smooth bounded open set $\omega$ such that $\overline\omega\subset\R^2\setminus ([0,\infty)\times \left\{0\right\}))$. Since $(\varphi-\theta)/(2\pi)$ is $\Z$-valued, Lemma \ref{Eunicite} yields that $\va-\theta$ is constant a.e. in $\Omega\setminus ([0,\infty)\times \left\{0\right\})$. Thus, $\theta\in B^s_{p,q}(\Omega)$. Similarly, $\widetilde\theta\in B^s_{p,q}(\Omega)$, where $\widetilde\theta
\in C^{\infty}(\R^2\setminus ((-\infty,0]\times \left\{0\right\}))$ is such that $e^{\im\widetilde\theta}=u$. We find that $(\theta-\widetilde\theta)/(2\pi)\in B^s_{p,q}(\Omega)$. However, this  is a non constant integer-valued function. This contradicts Lemma \ref{Eunicite} and proves non existence of lifting in $B^s_{p, q}$.
  
%
%

When $n\ge 3$, the above arguments lead to the following. Let $u(x)=\d\frac{(x_1, x_2)}{|(x_1, x_2)|}$, and let $\Omega\subset \R^n$ be a smooth bounded domain. Then $u\in B^s_{p, q}(\Omega ; \so)$ and, if $0\in\Omega$, then $u$ has no $B^s_{p, q}$ lifting.
\ep

 \section{Open cases}
 \label{ope}
 \begin{case}
\l{xa1}
{\it Range.} $s>1$, $1\le p<\infty$, $p<q<\infty$, $n\ge 3$, and $sp=2$. 

\smallskip
\noindent
{\it Discussion.}
This case is complementary to Case \ref{Y}. In the above range, we conjecture that the conclusion of Case \ref{Y} still holds, i.e., that the space $B^s_{p,q}(\Omega ; \so)$ does not have the lifting property. The non restriction property (Proposition \ref{l7.26}) prevents us from extending the argument used in Case \ref{Y} to Case \ref{xa1}. 
\end{case}
 \begin{case}
\l{Z}
{\it Range.} $s=1$, $1\le p<\infty$, $1\le q<\infty$, $n\ge 3$, and $2\le p<n$. 

Or: $s=1$, $1\le p<\infty$, $ q=\infty$, $n\ge 3$, and $2< p\le n$.

\smallskip
\noindent
{\it Discussion.}
When $p=q=2$, $B^1_{2,2}(\Omega ; \so)=H^1(\Omega ; \so)$ does have the lifting property \cite[Lemma 1]{bethuelzheng}. The remaining cases are open. The major difficulty arises from the extension of Lemma \ref{at3} to the range considered in Case \ref{Z}.
\end{case}

\begin{case}
\l{T}
{\it Range.} $s=0$, $1\le p<\infty$, $1\le q<\infty$ (and arbitrary $n$).

\smallskip
\noindent
{\it Discussion.} As explained in Remark \ref{aa1}, we consider only measurable functions $u:\Omega\to\so$. We let $B^0_{p, q}(\Omega ; \so):=\{ u:\Omega\to\so;\, u \text{ measurable and }u\in B^0_{p,q}\}$, and for $u$ in this space we are looking for a phase $\va\in L^1_{loc}\cap B^0_{p,q}$. 

Note that $B^0_{p,\infty}(\Omega ; \so)$ does have the lifting property. Indeed, in this case we have $L^\infty\subset B^0_{p,\infty}$ (Lemma \ref{ia2}) and then it suffices to argue as in the proof of Case \ref{aa1}. More generally, $B^0_{p,q}(\Omega ; \so)$ has the lifting property when $L^\infty\ho B^0_{p,q}$.\footnote{ A special case of this is $p=q=2$, since $B^0_{2,2}=L^2$. Another special case is $1<p\le 2\le q$. Indeed, in that case we have  $L^\infty\ho L^p=F^0_{p,2}\ho B^0_{p,q}$ \cite[Section 2.3.5, p. 51]{triebel2}, \cite[Section 2.3.2, Proposition 2, p. 47]{triebel2}.} The remaining cases are open.
\end{case}
\begin{case}
\l{xa3}
{\it Range.}
$0<s\le 1$, $p=1/s$, $q=\infty$ (and arbitrary $n$).

\smallskip
\noindent
{\it Discussion.}
We do not know whether $B^s_{p,q}(\Omega ; \so)$ does have the lifting property.
\end{case}
\begin{case}
\l{xa4}
{\it Range.}
$0<s\le 1$, $1< p<\infty$, $q=\infty$, $n\ge 3$, and $sp=n$.

\smallskip
\noindent
{\it Discussion.}
We do not know whether $B^s_{p,q}(\Omega ; \so)$ does have the lifting property. The difficulty common to Cases \ref{xa3} and \ref{xa4} is that in these ranges $B^s_{p,\infty}\not\subset\VMO$, and thus we are unable to rely on the strategy used in Cases \ref{X} and \ref{kc3}.
\end{case}
%

\section{Analysis in Besov spaces}
${}$

The results we state here are valid when $\Omega$ is a smooth bounded domain in $\R^n$, or $(0,1)^n$ or $\T^n$. However, in the proofs we will consider only one of these sets, the most convenient for the proof.

\subsection{Embeddings}
\l{ape}  
${}$
        
\begin{lemma}\label{Besovemb}  Let $0<s_1<s_0<\infty$, $1\le p_0<\infty$, $1\le p_1< \infty$, $1\le q_0\leq \infty$ and $1\le q_1\leq \infty$. Then the following hold.
\begin{enumerate}
\item
If $q_0<q_1$, then $B^s_{p,q_0}\ho B^s_{p,q_1}$.
\item If $s_0-n/p_0=s_1-n/p_1$, then  $ B^{s_0}_{p_0,q_0}    \hookrightarrow B^{s_1}_{p_1,q_0}$.     
\item If $s_0-n/p_0>s_1-n/p_1$, then  $ B^{s_0}_{p_0,q_0}    \hookrightarrow B^{s_1}_{p_1,q_1}$.   
\item  If  $B^{s_0}_{p_0,q_0}    \hookrightarrow B^{s_1}_{p_1,q_1}$,   then $s_0-n/p_0\geq s_1-n/p_1$.
\end{enumerate}

Consequently, when $q_0\leq q_1$,
\begin{equation} \label{equiv}
  B^{s_0}_{p_0,q_0}    \hookrightarrow B^{s_1}_{p_1,q_1} \iff s_0-\frac{n}{p_0}\geq s_1-\frac{n}{p_1}.
  \end{equation}
\end{lemma}
\begin{proof} For item 1, see \cite[Section 3.2.4]{triebel2}.  For items  2 and  3, see \cite[Section 3.3.1]{triebel2} or  \cite[Theorem 1, p. 82]{runstsickel}. Item 4 follows from a scaling argument. And  \eqref{equiv} is an immediate consequence of items 1--4.  
\end{proof}
For the next result, see e.g. \cite[Section 2.7.1, Remark 2, pp. 130-131]{triebel2}.
\bl
\l{ka1}
Let $s>0$, $1\le p<\infty$, $1\le q\le\infty$ be such that $sp>n$. Then $B^s_{p,q}(\Omega)\ho C^0(\overline\Omega)$.
\el
 \bl
 \l{ia2}
 Let $s< 0$, $1\le p<\infty$ and $1\le q\le \infty$. Then $L^\infty\ho B^s_{p,q}$.
 
 Similarly, if $1\le p\le\infty$, then $L^\infty\ho B^0_{p,\infty}$.
 \el     
 \bp
 We present the argument when $\Omega=\T^n$. 
 Let $f\in L^\infty$, with Fourier coefficients $(a_m)_{m\in\Z^n}$. Consider, as in Definition \ref{periodicbesov}, the functions
 \bes
f_j(x):=\sum_{m\in\Z^n}a_m\va_j(2\pi m)\, e^{2\im \pi m\cdot x},\ \fo j\in\N.
 \ees
 By the (periodic version of)  the multiplier theorem \cite[Section 9.2.2, Theorem, p. 267]{triebel2} we have
 \be
 \l{kb1}
 \|f_j\|_{L^p}\lesssim \|f\|_{L^p},\ \fo 1\le p\le \infty,\ \fo j\in\N. 
 \ee
 We find that $\|f_j\|_{L^p}\lesssim \|f\|_{L^p}\le \|f\|_{L^\infty}$, and thus (by Definition \ref{periodicbesov}, and with the usual modification when $q=\infty$)
 \bes
 \|f\|_{B^s_{p,q}}\lesssim \left(\sum_j 2^{sjq}\right)^{1/q}<\infty.
 \ees
 The second part of the lemma follows from a similar argument. The proof is left to the reader.
 \ep

 An analogous proof  leads to the following result. Details are left to the reader.
 \bl
 \l{kc2}
 Let $s>0$, $1\le p<\infty$ and $1\le q\le\infty$. Then $B^s_{p,q}\ho L^p$. 
 
 More generally, if $k\in\N$, $s>k$, $1\le p<\infty$, and $1\le q\le \infty$, then $B^s_{p,q}\ho W^{k, p}$.
 \el
 \begin{lemma}\label{B-VMO}  Let $0<s<\infty$, $1\le p<\infty$ and $1\le q<\infty$ be such that $sp=n$. Then
$\d B^{s}_{p,q}   \hookrightarrow  \VMO$.\\
Same conclusion if $0<s<\infty$, $1\le p<\infty$ and $q=\infty$ are such that $sp>n$.
\end{lemma}  
\begin{proof}
Assume first that $q<\infty$.  Let $p_1>\max\left\{n,p,q\right\}$ and set $s_1:=n/{p_1}$. By Lemma \ref{Besovemb} and the fact that $s_1$ is not an integer, we have
\bes
 B^s_{p,q}\hookrightarrow B^{s_1}_{p_1,q}\hookrightarrow B^{s_1}_{p_1,p_1}=W^{s_1,p_1}.
\ees
 It then suffices to invoke the embedding
\bes
 W^{s_1,p_1}\hookrightarrow \VMO\mbox{ when }s_1p_1=n
 \ \text{\cite[Example 2, p. 210]{brezisnirenberg1}}.\ees    
 The case  where  $q=\infty$ is obtained via the first part of the proof. Indeed, it suffices to choose   $0<s_1<\infty$, $1\le p_1<\infty$ and $0<q_1<\infty$   such that $s_1p_1=n$ and $B^s_{p,q}\ho B^{s_1}_{p_1,q_1}$. Such $s_1$, $p_1$ and $q_1$ do exist, by Lemma \ref{Besovemb}.
 \end{proof}

 For the following special case of the Gagliardo-Nirenberg embeddings, see e.g.  \cite[Remark 1, pp. 39-40]{runstsickel}.
 \begin{lemma} \label{gn}
 Let $0<s<\infty$, $1\le p<\infty$, $1\le q\leq \infty$, and $0<\theta<1$. Then $B^s_{p,q}\cap L^{\infty}\hookrightarrow B^{\theta s}_{p/\theta,q/\theta}$.
 \end{lemma}

\subsection{Restrictions}
\l{mo6}
${}$

{\it Captatio benevolenti\ae}. Let $f\in L^1(\R^2)$. Then, for a.e., $y\in\R$, the restriction  $f(\cdot, y)$ of $f$ to the line $\R\times\{y\}$ belongs to $L^1$. In this section and the next one, we examine some analogues of this property in the framework of Besov spaces. 

For this purpose, we first introduce some notation for partial functions. 
Let $\alpha\subset \{1,\ldots, n\}$ and set $\overline\alpha:=\{1,\ldots, n\}\setminus \alpha$. If $x=(x_1,\ldots, x_n)\in\R^n$, then we identify $x$ with the couple $(x_\alpha, x_{\overline\alpha})$, where $x_\alpha:=(x_j)_{j\in\alpha}$ and $x_{\overline\alpha}:=(x_j)_{j\in\overline\alpha}$. 
Given a function $f=f(x_1,\ldots, x_n)$, we let $f_\alpha=f_\alpha(x_\alpha)$ denote the partial function $x_{\overline\alpha}\mapsto f(x)$. 
Another useful notation: given an integer $m$ such that $1\le m\le n$, set
\bes
I(n-m,n):=\{\alpha \subset \{1,\ldots, n\};\, \#\alpha=n-m\}.
\ees
Thus, when $\alpha\in I(n-m,n)$,  $f_\alpha(x_\alpha)$ is a function of $m$ variables. 

When $q=p$, we have the following result.
\bl
\l{oa1}
Let $1\le m<n$. Let $s>0$ and $1\le p<\infty$. Let $f\in B^s_{p,p}(\R^n)$.
\ben
\item
Let $\alpha\in I(n-m,n)$. Then, for a.e. $x_\alpha\in\R^{n-m}$, we have $f_\alpha(x_\alpha)\in B^s_{p,p}(\R^m)$.
\item
We have
\bes
\|f\|_{B^s_{p,p}(\R^n)}^p\sim  \sum_{\alpha\in I(n-m,n)}\int_{\R^{n-m}}\|f_\alpha(x_\alpha)\|_{B^s_{p,p}(\R^m)}^p\, dx_\alpha. 
\ees
\een
\el
\bp
For the case where $m=1$, see 
 \cite[Section 2.5.13,  Theorem, (i), p. 115]{triebel2}. The general case is obtained by a straightforward induction on $m$.
 \ep
 \bl
\l{mo7}
Let $s>0$, $1\le p<\infty$ and $1\le q\le p$. Let $1\le m< n$ be an integer. Assume that $sp\ge m$ and let $f\in B^s_{p, q}(\T^n)$. Then, for every $\alpha\in I(n-m,n)$ and for a.e. $x_\alpha\in\T^{n-m}$, the partial map $f_\alpha(x_\alpha)$ belongs to $\VMO(\T^m)$.

Same conclusion if $s>0$, $1\le p<\infty$ and $1\le q\le \infty$, and we have $sp>m$. 

Similar conclusions when $\Omega=\R^n$ or $(0,1)^n$.
\el
\bp
In view of the Sobolev embeddings (Lemma \ref{Besovemb}), we may assume that $sp=m$ and $q=p$. By Lemma \ref{oa1} and Lemma \ref{B-VMO}, for a.e. $x_\alpha$ we have $f_\alpha(x_\alpha)\in B^s_{p,p}(\T^m)\ho \VMO(\T^m)$.
\ep
\bl
\l{ad1}
Let $s>0$, $1\le p<\infty$ and $1\le q<\infty$. Let $M>s$ be an integer. Let  $f\in B^s_{p,q}$.  For $x'\in \T^{n-1}$, consider the partial map $v(x_n)=v_{x'}(x_n):=f(x',x_n)$, with $x_n\in\T$. Then there exists a sequence  $(t_l)\subset (0,\infty)$ such that $t_l\to 0$ and for a.e. $x'\in\T^{n-1}$, we have
\be
\l{ce1}
\lim_{l\to\infty}\frac{\left\|\Delta_{t_l}^Mv\right\|_{L^p(\T)}}{t_l^{s}}=0.
\ee
More generally, given a finite number of functions $f_j\in B^{s_j}_{p_j,q_j}$, with $s_j>0$, $1\le p_j<\infty$ and $1\le q_j<\infty$, and given an integer $M>\max_j s_j$,  we may choose a common set $A$ of full measure in $\T^{n-1}$ and a  sequence $(t_l)$ such that the analog of  \eqref{ce1}, i.e., 
\be
\l{cf1}
\lim_{l\to\infty}\frac{\left\|\Delta_{t_l}^Mf_j(x',\cdot)\right\|_{L^{p_j}(\T)}}{ t_l^{s_j}}=0,
\ee
 holds simultaneously for all  $j$ and all $x'\in A$.
\el
\bp
We treat the case of a single function; the general case is similar.

Set $g_t:=\left\|\Delta^M_{t e_n}f\right\|_{L^p}$. By \eqref{equivnormhomogrn}, we have $\int_0^1t^{-sq-1}g_t^q\, dt<\infty$, which is equivalent to 
$
\int_{1/2}^1\sum_{m\ge 0}2^{msq}g_{2^{-m}\sigma}^q\, d\sigma<\infty$. Therefore, there exists some $\sigma\in (1/2,1)$ such that 
\be
\l{ce2}
\sum_{m\ge 0}2^{msq}g_{2^{-m}\sigma}^q <\infty. 
\ee
By \eqref{ce2} , we find that
\be
\l{ce3}
\lim_{m\to \infty}\frac{g_{2^{-m}\sigma}}{(2^{-m}\sigma)^s}=0.
\ee
Using \eqref{ce3} we find that, along a subsequence $(m_l)$, we have 
\bes
\lim_{m\to \infty}\frac{\|\Delta_{2^{-m_l}\sigma}v\|_{L^p}}{(2^{-m_l}\sigma)^s}=0\quad\text{for a.e. }x'\in\T^{n-1}.
\ees 
This implies \eqref{ce1} with $t_l:=2^{-m_l}\sigma$.
\ep

%
%
          
    
    \subsection{(Non) restrictions}

${}$

We now address the question whether, given $f\in B^s_{p, q}(\R^2)$, we have $f(x, \cdot)\in B^s_{p, q}(\R)$ for a.e. $x\in\R$. This kind of questions can also be asked in higher dimensions. The answer crucially depends on the sign of $q-p$.

 We start with a simple result.
 
\bpr
\l{qh1}
Let $s>0$ and $1\le q\le p<\infty$. Let $f\in B^s_{p,q}(\R^2)$. Then for a.e. $x\in\R$ we have $f(x, \cdot)\in B^s_{p,q}(\R)$. 
\epr
 \bp
 Let $f\in B^s_{p,q}(\R^2)$. Using \eqref{equivnormhomogrn} (part 2) and H\"older's inequality, we find that for every finite interval $[a,b]\subset\R$ and $M>s$ we have
 \bes
 \begin{aligned}
 \int_a^b |f(x,\cdot)|_{B^s_{p,q}(\R)}^q\, dx& \sim  
 \int_a^b\int_\R \frac 1{|h|^{sq+1}}\left(\int_{\R}|\Delta^M_{h e_2}f(x, y)|^p\, dy\right)^{q/p}\, dh dx\\
 &\le (b-a)^{(p-q)/p}\, \int_\R \frac 1{|h|^{sq+1}}\left(\int_{[a,b]\times\R}|\Delta^M_{h e_2}f(x, y)|^p\, dxdy\right)^{q/p}\, dh\\
 &\lesssim |f|_{B^s_{p,q}(\R^2)}^q<\infty\end{aligned}
 \ees
 whence the conclusion.
  \ep
 
 When $q>p$, a striking phenomenon occurs.
\bpr
\l{l7.26}
Let $s>0$ and $1\le p<q\le\infty$. Then there exists some compactly supported $f\in B^s_{p,q}(\R^2)$ such that for a.e. $x\in (0,1)$ we have $f(x,\cdot)\not\in B^s_{p,\infty}(\R)$.

 In particular, for any $1\le r<\infty$ and a.e. $x\in (0,1)$ we have $f(x,\cdot)\not\in B^s_{p, r}(\R)$.
\epr

Before proceeding to the proof, let us note that if  $f\in B^s_{p,q}(\R^2)$ then $f\in L^p(\R^2)$, and thus the partial function $f(x,\cdot)$ is a well-defined element of $L^p(\R)$ for a.e. $x$.

\bp

Since $B^s_{p, q}(\R^2)\subset B^s_{p,\infty}(\R^2)$, $\fo q$, we may assume that $q<\infty$. 
 We rely on the characterization of Besov spaces in terms of smooth wavelets, as in Section \ref{qa6}. 

 We start by explaining the construction of $f$. Let $\psi_F$ and $\psi_M$ be as in Section \ref{qa6}. With no loss of generality, we may assume that $\supp\psi_M\subset [0,a]$ with $a\in\N$. Consider $(\alpha, \beta)\subset  (0,a)$ and $\gamma>0$ such that $\psi_M\ge\gamma$ in $[\alpha, \beta]$.
 
 Set $\delta:=\beta-\alpha>0$ and consider some integer $N$ such that $[0,1]\subset [\alpha-N\, \delta, \beta+N\, \delta]$. We look for an $f$ of the form
 \be
 \l{qb5}
 f=\sum_{\ell=-N}^N\sum_{j\ge j_0} g^\ell_j,
 \ee
 with 
 \be
 \begin{aligned}
 \l{qb6}
 g^\ell_j(x,y)=\mu_j\, 2^{-j(s-2/p)}\sum_{m_1\in I_j}  &\psi_M(2^j x-m_1-\ell\, \delta)\\
 &\times\psi_M(2^j y-m_1-2^{j+1}\, \ell\, a-\ell\, \delta).
 \end{aligned}
 \ee
 
 Here, the set $I_j$ satisfying
 \be
 \l{qb7}
 I_j\subset \{ 0, 1,\ldots, 2^j\},
 \ee
the integer $j_0$ and the  coefficients $\mu_j>0$ will be defined later. 

We consider the partial sums $f^\ell_J:=\sum_{j=j_0}^J g^\ell_j$. Clearly, we have $f^\ell_J\in C^k$ and, provided $j_0$ is sufficiently large,
\bes
\sup f^\ell_J\subset K_l:=[-N\, \delta , 5/4]\times [2\ell\, a-1/4, (2\ell +1)\, a+1/4].
\ees

We next note that the compacts $K_\ell$ are mutually disjoint. Using Proposition \ref{p2.4} item 2, we easily find that
\be
\l{qb9}
\left\|\sum_{\ell=-N}^N f^\ell_J  \right\|_{B^s_{p,q}(\R^2)}^q\sim  \sum_{\ell=-N}^N \left\| f^\ell_J  \right\|_{B^s_{p,q}(\R^2)}^q.
\ee

On the other hand, if $\psi_M$ and $\psi_F$ are wavelets such that Proposition \ref{qb1} holds, then so are $\psi_F(\cdot-\lambda)$ and $\psi_M(\cdot-\lambda)$, $\fo \lambda\in\R$ \cite[Theorem 1.61 {\it (ii)}, Theorem 1.64]{triebel06}. Combining this fact with \eqref{qb9}, we find that
\be
\l{qc1}
\left\|\sum_{\ell=-N}^N f^\ell_J  \right\|_{B^s_{p,q}(\R^2)}^q\sim \sum_{j=j_0}^J \left( \# I_j\, (\mu_j)^p\right)^{q/p}.
\ee

We now make the size assumption 
\be
\l{qc2}
\sum_{j=j_0}^\infty \left( \# I_j\, (\mu_j)^p\right)^{q/p}<\infty.
\ee

 By \eqref{qc1} and \eqref{qc2}, we see that the formal series in \eqref{qb5} defines a compactly supported $f\in B^s_{p,q}(\R^2)$, with $\sum_{\ell=-N}^N f^\ell_J\to f$ in $B^s_{p,q}(\R^2)$ (and therefore in $L^p(\R^2)$) as $J\to\infty$. 
 
  We next investigate the $B^s_{p,\infty}$ norm of the restrictions $f^\ell_J (x,\cdot)$. As in \eqref{qb9}, we have
  \be
  \l{qc3}
  \left\|\sum_{\ell=-N}^N f^\ell_J (x,\cdot)\right\|_{B^s_{p,\infty}(\R)}\sim \sum_{\ell=-N}^N\|f^\ell_J(x,\cdot)\|_{B^s_{p,\infty}(\R)}.
  \ee
  
  Rewriting \eqref{qb6} as
  \be
  \l{qc4}
  \begin{aligned}
  g^\ell_j(x,y)=\mu_j\, 2^{-j(s-1/p)}\, 2^{j/p}\sum_{m_1\in I_j} & \psi_M(2^j x-m_1-\ell\, \delta)\\
  &\times \psi_M(2^j y-m_1-2^{j+1}\, \ell\, a -\ell\, \delta),
  \end{aligned}
  \ee
we obtain
\be
\l{qc5}
\|f^\ell_J(x,\cdot)\|_{B^s_{p,\infty}(\R)}^p\sim \sup_{j_0\le j\le J} 2^j\, (\mu_j)^p\, \sum_{m_1\in I_j}|\psi_M (2^j\, x- m_1-\ell\, \delta)|^p.
\ee

We now make the size assumption
\be
\l{qc6}
\sup_{j\ge j_0} 2^j\, (\mu_j)^p\, \sum_{\ell=-N}^N\sum_{m_1\in I_j}|\psi_M (2^j\, x- m_1-\ell\, \delta)|^p=\infty,\ \fo x\in [0,1].
\ee

Then we claim that for a.e. $x\in (0,1)$ we have 
\be
\l{qc7}
f(x,\cdot)\not\in B^s_{p,\infty}(\R).
\ee

 Indeed, since $\sum_{\ell=-N}^N f^\ell_J\to f$ in $L^p (\R^2)$,  for a.e. $x\in\R$ we have
\be
\l{qc8}
\sum_{\ell=-N}^\ell f^\ell_J(x,\cdot)\to f(x,\cdot)\ \text{in }L^p(\R).
\ee

We claim that for every $x\in [0,1]$ such that \eqref{qc8} holds, we have $f(x,\cdot)\not\in B^s_{p,\infty}(\R)$. Indeed, on the one hand \eqref{qc6} implies that for some $\ell$ we have $\lim_{J\to\infty}\|f^\ell_J (x,\cdot)\|_{B^s_{p,\infty}(\R)}=\infty$. We assume e.g. that this holds when $\ell=0$. Thus 
\be
\l{qc80}
\sup_{j\ge j_0} 2^j\, (\mu_j)^p\, \sum_{m_1\in I_j}|\psi_M (2^j\, x- m_1)|^p=\infty.
\ee

On the other hand, assume by contradiction that $f(x,\cdot)\in B^s_{p,\infty}(\R)$. Then we may write $f(x,\cdot)$ as in \eqref{qb3}, with coefficients as in \eqref{qb40}. In particular, taking into account the explicit formula of $g^\ell_j$ and the fact that $\sum_{\ell=-N}^N f^\ell_J(x,\cdot)\to f(x,\cdot)$ in $L^p(\R)$, we find that for $k\ge j_0$ and $m_1\in I_j$ we have  
\be
\l{qc800}
\begin{aligned}
\mu^{k, \{ M\}}_{m_1}(f(x,\cdot))&=\mu^{k, \{ M\}}_{m_1}\left(\sum_{j=j_0}^J g^0_j (x, \cdot)\right)=\mu^{k, \{ M\}}_{m_1}(g^0_k (x,\cdot))\\
&=2^{k/p}\, \mu_k\, \psi_M(2^k\, x-m_1),\ \fo J\ge k.
\end{aligned}
\ee

We obtain a contradiction combining \eqref{qc80}, \eqref{qc800} and Corollary \ref{qb400}.

It remains to construct $I_j$ and $\mu_j$ satisfying \eqref{qb7}, \eqref{qc2} and \eqref{qc6}. We will let $I_j=\llbracket  s_j,  t_j\rrbracket$, with $0\le s_j\le t_j\le 2^j$ integers to be determined later. Set $t:=q/p \in (1,\infty)$ and
\bes
\mu_j:=\left(\frac 1{(t_j-s_j+1)\, j^{1/t}\, \ln j}\right)^{1/p}.
\ees

Clearly, \eqref{qb7} and \eqref{qc2} hold. It remains to define $I_j$ in order to have \eqref{qc6}. Consider the dyadic segment $L_j:=[s_j/2^j, t_j/2^j]$. We claim that 
\be
\l{qa11}
\sum_{\ell=-N}^N\sum_{m_1\in I_j}|\psi_M (2^j\, x- m_1-\ell\, \delta)|^p\ge \gamma^p,\ \fo x\in L_j. 
\ee

Indeed, let $m_1\in [s_j, t_j]$ be the integer part of $2^j\, x$. By the definition of $\delta$ and by choice of $N$, there exists some $\ell\in \llbracket -N, N\rrbracket$ such that $\alpha\le 2^j\, x- m_1-\ell\, \delta\le \beta$, whence the conclusion. 

 By the above, \eqref{qc6} holds provided we have
 \be
 \l{qc60}
 \sup_{j\ge j_0}2^j\, (\mu_j)^p\, \one_{L_j(x)}=\infty,\ \fo x\in [0,1].
 \ee
 
  We next note that 
  \be
  \l{qc600}
  2^j\, (\mu_j)^{p}\sim \frac 1{|L_j|\, j^{1/t}\, \ln j}=\frac {u_j}{|L_j|},
  \ee
  where $u_j:=1/(j^{1/t}\, \ln j)$ satisfies 
  \be
  \l{qc6000}
  \sum_{j\ge j_0}u_j=\infty.
  \ee
  
  In view of \eqref{qc600} and \eqref{qc6000}, existence of $I_j$ satisfying \eqref{qc60} is a consequence of Lemma \ref{tempSeq} below. The proof of Proposition \ref{l7.26} is complete.\ep

\begin{lemma}\label{tempSeq}
Consider a sequence $(u_j)$ of positive numbers such that $\sum_{j\ge j_0}u_j=\infty$. Then there exists a sequence $(L_j)$ of dyadic intervals $L_j=[s_j/2^j, t_j/2^j]$, such that:
\ben
\item
$s_j, t_j\in\N$, $0\le s_j< 2^j$.
\item
 $|L_j|=o(u_j)$ as $j\to\infty$.
 \item
 Every $x\in [0,1]$ belongs to infinitely many $L_j$'s.
 \een
\end{lemma}
\bp
Consider a sequence $(v_j)$ of positive numbers such that $\sum_{j\ge j_0}v_j\, u_j=\infty$ and $v_j\to 0$. Let $L_{j_0}$ be the largest dyadic interval of the form $[0, t_{j_0}/2^{j_0}]$ of length $\le v_{j_0}\, u_{j_0}$. This defines $s_{j_0}=0$ and $t_{j_0}$.

 Assuming $L_j=[s_j/2^j, t_j/2^j]=[a_j, b_j]$ constructed for some $j\ge j_0$, one of the following two occurs. Either $b_j<1$ and then we let $L_{j+1}$ be the largest dyadic interval of the form $[2 t_{j}/2^{j+1}, t_{j+1}/2^{j+1}]$ such that $|L_{j+1}|\le v_{j+1}\, u_{j+1}$. Or $b_j\ge 1$, and then we let $L_{j+1}$ be the largest dyadic interval of the form $[0, t_{j+1}/2^{j+1}]$  such that $|L_{j+1}|\le v_{j+1}\, u_{j+1}$.
 
 Using the assumption $\sum_{j\ge j_0}v_j\, u_j=\infty$ and the fact that $|L_j|\ge v_j\, u_j-2^{-j}$, we easily find that for every $j\ge j_0$ there exists some $k>j$ such that $L_k=[a_k, b_k]$ satisfies $b_k\ge 1$, and thus the intervals $L_j$ cover each point $x\in [0,1]$ infinitely many times. \ep
 
\br
\l{r10}
Following a suggestion of the first author, Brasseur  investigated the non restriction property established in Proposition \ref{l7.26}. In \cite{brasseur} (which is independent of the present work), Brasseur extends Proposition \ref{l7.26} to the full range $0<p<q\le \infty$; the construction is somewhat similar to ours (based on the size of the coefficients $\mu_j$ in the decomposition \eqref{qb6}), but relying on a  different decomposition (subatomic instead of wavelets). \cite{brasseur} also contains an interesting positive result: it exhibits function spaces $X$ intermediate between $B^s_{p,q}(\R)$ and $\d\bigcup_{\ve>0}B^{s-\ve}_{p, q}(\R)$  such that, if $f\in B^s_{p,q}(\R^2)$, then for a.e. $x\in\R$ we have $f(x,\cdot)\in X$.
\er

\subsection{Poincar\'e type inequalities}  
${}$

The next  Poincar\'e type inequality for Besov spaces  is certainly well-known, but we were unable to find a reference in the literature.       
\begin{lemma}
\l{ad2}
Let $0<s<1$, $1\leq p<\infty$, and $1\le q\le\infty$.  Then we have 
\begin{equation} \label{PBesov}
\left\Vert f-\fint f\right\Vert_{L^p}\lesssim \left\vert f\right\vert_{B^{s}_{p,q}},\quad \fo f:\Omega\to\R\text{ measurable function}.
\end{equation}
\end{lemma}
\noindent
Recall (Proposition \ref{p2.4}) that the
 semi-norm in \eqref{PBesov} is given by 
\be
\l{aa4}
|f|_{B^s_{p,q}}=|f|_{B^s_{p,q}(\R^n)}:=\left(\int_{\R^n}|h|^{-sq}\|\Delta_h f\|_{L^p}^q\, \frac{dh}{|h|^n}\right)^{1/q}
\ee
when $q<\infty$, 
with the obvious modifications when $q=\infty$ or $\R^n$ is replaced by $\Omega$. 
\begin{proof} 
%
By \eqref{homoglp}, we have $\|f\|_{B^s_{p,q}}\sim \|f\|_{L^p}+|f|_{B^s_{p,q}}$. 
Recall that the embedding $B^{s}_{p,q}\hookrightarrow L^p$ is compact \cite[Theorem 3.8.3, p. 296]{triebel1}. From this we infer that  \eqref{PBesov} holds for every function $f\in B^{s}_{p,q}$. Indeed, assume by contradiction that this is not the case. Then there exists a sequence of functions $(f_j)_{j\geq 1}\subset B^s_{p,q}$ such that, for every $j$,
\bes
1=\left\Vert f_j-\fint f_j\right\Vert_{L^p}\geq j \left\vert f_j\right\vert_{B^{s}_{p,q}}. \Bk
\ees
Set $g_j:={f_j-\fint f_j}$. 
Then, up to a subsequence, we have $g_j\to g$ in $L^p$, where $\|g\|_{L^p}=1$ and $\int g=0$. 
We claim that $g$ is constant in $\Omega$ (and thus $g=0$). Indeed, by the Fatou lemma, for every $h\in \R^n$ we have
\be
\l{aa3}
\|\Delta_hg\|_{L^p}\le \liminf \|\Delta_hg_j\|_{L^p}=\liminf \|\Delta_hf_j\|_{L^p}.
\ee
By \eqref{aa4}, \eqref{aa3} and the Fatou lemma, we have
\bes
|g|_{B^s_{p,q}}\le\liminf |g_j|_{B^s_{p,q}}=\liminf |f_j|_{B^s_{p,q}}=0;
\ees
thus $g=0$, as claimed. This contradicts the fact that $\|g\|_{L^p}=1$.

 Let us now establish \eqref{PBesov} only assuming that $|f|_{B^s_{p,q}}<\infty$. We start by reducing the case where $q=\infty$ to the case where $q<\infty$. This reduction relies on the straightforward estimate
 \bes
 |f|_{B^\sigma_{p,r}}\lesssim |f|_{B^s_{p,\infty}},\quad \fo 0<\sigma<s,\ \fo 0<r<\infty.
 \ees
So let us assume that $q<\infty$. 
%
%
For every integer $k\geq 1$, let $\Phi_k:\R\rightarrow \R$ be given by
\bes
\Phi_k(t):=
\begin{cases}
t, & \mbox{ if }\left\vert t\right\vert \leq k\\
-k, & \mbox{ if }t\leq -k\\
k, & \mbox{ if }t\geq k
\end{cases}.
\ees
Clearly, $\Phi_k$ is $1$-Lipschitz, so that \eqref{aa4} easily yields 
\begin{equation} \label{controlphik}
\left\vert \Phi_k(f)\right\vert_{B^s_{p,q}}\leq \left\vert f\right\vert_{B^s_{p,q}}
\end{equation}
and (by dominated convergence, using $q<\infty$  and \eqref{aa4})
\begin{equation} \label{convphikf}
\lim_{k\rightarrow \infty} \left\vert \Phi_k(f)-f\right\vert_{B^{s}_{p,q}}=0.
\end{equation}
Since $\Phi_k(f)\in L^{\infty}(\Omega)\subset L^p(\Omega)$, one has  $\Phi_k(f)\in B^s_{p,q}$ for every $k$. Therefore, \eqref{PBesov} and \eqref{controlphik} imply
\begin{equation} \label{phikck}
\left\Vert \Phi_k(f)-c_k\right\Vert_{L^p}\lesssim \left\vert \Phi_k(f)\right\vert_{B^{s}_{p,q}}\le \left\vert f\right\vert_{B^s_{p,q}}
\end{equation}
with $c_k:=\fint \Phi_k(f)$.  Thanks to \eqref{convphikf}, we may pick up an increasing sequence of integers $(\lambda_k)_{k\geq 1}$ such that, for every $k$,
$\d
\left\vert \Phi_{\lambda_{k+1}}(f)-\Phi_{\lambda_k}(f)\right\vert_{B^s_{p,q}}\leq 2^{-k}$.
Applying \eqref{PBesov} to $ \Phi_{\lambda_{k+1}}(f)-\Phi_{\lambda_k}(f)$, one therefore has
\bes
\left\Vert \left(\Phi_{\lambda_{k+1}}(f)-c_{\lambda_{k+1}}\right)-\left(\Phi_{\lambda_k}(f)-c_{\lambda_k}\right)\right\Vert_{L^p}\lesssim \left\vert \Phi_{\lambda_{k+1}}(f)-\Phi_{\lambda_k}(f)\right\vert_{B^s_{p,q}}\leq 2^{-k},
\ees
which entails that $\d  
\Phi_{\lambda_k}(f)-c_{\lambda_k}\to g\text{ in }L^p$  as $k\to\infty$.
Up to a subsequence, one can also assume that $
\d 
\Phi_{\lambda_k}(f)(x)-c_{\lambda_k}\to g(x)$ for a.e. $x\in \Omega$.
Take any $x\in \Omega$ such that $\Phi_{\lambda_k}(f)(x)-c_{\lambda_k}\to g(x)$. 
Since
$\d
\Phi_{\lambda_k}(f)(x)\to f(x)$ as $k\to\infty$,
one obtains
\begin{equation} \label{ckc}
\lim_{k\rightarrow \infty} c_{\lambda_k}=c\in \C.
\end{equation}
Finally, \eqref{phikck}, \eqref{ckc} and the Fatou lemma yield
$\d
\left\Vert  f-c\right\Vert_{L^p}\lesssim \left\vert f\right\vert_{B^s_{p,q}}$,
from which \eqref{PBesov} easily follows.
\end{proof}

We next state and prove a generalization of Lemma \ref{ad2}. 
\bl
\l{ad3}
Let $0<s<1$, $1\le p<\infty$, $1\le q\le\infty$, and $\delta\in (0,1]$. Define
\be
\l{ad4}
|f|_{B^s_{p,q,\delta}}:=\left(\int_{|h|\le\delta}|h|^{-sq}\|\Delta_h f\|_{L^p}^q\, \frac{dh}{|h|^n}\right)^{1/q}
\ee
when $q<\infty$, 
with the obvious modifications when $q=\infty$ or $\R^n$ is replaced by $\Omega$. Then we have
\begin{equation} \label{ad5}
\left\Vert f-\fint f\right\Vert_{L^p}\lesssim \left\vert f\right\vert_{B^{s}_{p,q, \delta}},\quad \fo f:\Omega\to\R\text{ measurable function}.
\end{equation}
\el
\bp
 Recall that \Bk $\|f\|_{B^s_{p,q}}\sim \|f\|_{L^p}+|f|_{B^s_{p,q,\delta}}$ (Proposition \ref{p2.4}). We continue as in the proof of Lemma \ref{ad2}.
\ep
We end with an estimate involving derivatives.
\bl
\l{at4}
Let $s>0$, $1< p<\infty$ and $1\le q\le\infty$. 
Let $f\in {\cal D}'(\Omega)$ be such that $\na f\in B^{s-1}_{p,q}(\Omega)$. Then $f\in B^s_{p, q}(\Omega)$ and
\be
\l{at9}
\left\| f-\fint f\right\|_{B^s_{p,q}}\lesssim \|\na f\|_{B^{s-1}_{p,q}}. 
\ee
\el
The above result is well-known, but we were unable to find it in the literature; for the convenience of the reader, we present the short argument when $\Omega=\T^n$. 
\bp
We use the notation in Proposition \ref{mm2} and  the following result \cite[Lemma 2.1.1, p. 16]{chemin}:  we have 
\be
\l{mm3}
\|f_j\|_{L^p}\sim 2^{-j}\|\na f_j\|_{L^p},\quad \fo 1\le p\le \infty,\ \fo j\ge 1.
\ee
By combining \eqref{mm3} with Proposition \ref{mm2}, we obtain, e.g. when $q<\infty$:
\be
\l{mm4}
\begin{aligned}
\left\|f-a_0\right\|_{B^s_{p,q}}^q&=\left\|\sum_{j\ge 1}f_j\right\|_{B^s_{p,q}}^q\sim \sum_{j\ge 1}2^{sjq}\|f_j\|_{L^p}^q
\\
&\lesssim \sum_{j\ge 1}2^{sjq}2^{-jq}\|\na f_j\|_{L^p}^q
\sim \|\na f\|_{B^{s-1}_{p,q}}^q.
\end{aligned}
\ee
In particular, $f\in L^1$ (Lemma \ref{kc2}), and thus $a_0=\fint f$. Therefore, \eqref{mm4} is equivalent to \eqref{at9}.
\ep
\br
\l{mn41}
With more work, Lemma \ref{at4} can be extended to the case where $p=1$. Although this will not be needed here, we sketch below the argument. With the notation   in Section \ref{mm6}, consider the Littlewood-Paley decomposition  $f=\sum f_j$, with $f_j:=\sum a_m\va_j(2\pi m)e^{2\im\pi m\cdot x}$. Note that
the Littlewood-Paley decomposition of $\na f$ is simply given by
\be
\l{mn7}
\na f=\sum \na f_j.
\ee
In the spirit of \cite[Lemma 2.1.1, p. 16]{chemin} (see also \cite[Proof of Lemma 1]{leta}), one may prove that we have the following analog of \eqref{mm3}:
\be
\l{mn6}
\|f_j\|_{L^p}\sim 2^{-j}\|\na f_j\|_{L^p},\quad \fo 1\le p\le \infty,\ \fo j\ge 1.
\ee
Using Definition \ref{periodicbesov}, \eqref{mn7} and \eqref{mn6}, we obtain \eqref{mm4}. We conclude as in the proof of Lemma \ref{at4}.
\er
\subsection{Characterization of $B^s_{p,q}$ via extensions} \label{characext}
${}$

The type of results we present in this section are classical for functions defined on the whole $\R^n$ and for the harmonic extension. Such results were obtained by Uspenski\u\i{ } in the early sixties \cite{uspenskii}. For further developments, see \cite[Section 2.12.2,  Theorem, p. 184]{triebel2};  see also  Section \ref{chha}. When the harmonic extension is replaced by other extensions by regularization, the kind of results we present below were known to experts at least for maps defined on $\R^n$; see  \cite[Section 10.1.1, Theorem 1, p. 512]{mazyanew} and also \cite{tracesoldnew} for a systematic treatment of extensions by smoothing.
The local variants (involving extensions by averages in domains) we present below  could be obtained by adapting the arguments we developed in  a more general setting in \cite{tracesoldnew}, and which are quite involved. However, we present here a more elementary approach, inspired by \cite{mazyanew}, sufficient to our purpose.
In what follows, we let $|\ |$ denote the $\|\ \|_{\infty}$ norm in $\R^n$.  

For simplicity, we state our results when $\Omega=\T^n$, but they can be easily adapted to arbitrary $\Omega$.

\bl
\l{ab1}
Let $0<s<1$, $1\le p<\infty$,  $1\le q\le\infty$, and $\delta\in (0,1]$. Set 
$
V_\delta:=\T^n\times (0,\delta)$.
\ben
\item
Let $F\in C^\infty(V_{\delta})$. If
\be
\l{cg6}
\left(\int_0^{\delta/2}\ve^{q-sq}\|(\na F)(\cdot,\ve)\|_{L^p}^q\, \frac{d\ve}\ve\right)^{1/q}<\infty
\ee
(with the obvious modification when $q=\infty$), then $F$ has a trace $f\in B^s_{p,q}(\T^n)$, satisfying
\be
\l{ab2}
|f|_{B^s_{p,q,\delta}}\lesssim \left(\int_0^{\delta/2}\ve^{q-sq}\|(\na F)(\cdot,\ve)\|_{L^p}^q\, \frac{d\ve}{\ve}\right)^{1/q}.
\ee
\item
Conversely, let $f\in B^s_{p,q}(\T^n)$. Let $\rho\in C^\infty$ be a mollifier supported in $\{ |x|\le 1\}$ and set $F(x,\ve):=f\ast\rho_\ve(x)$, $x\in\T^n$, $0<\ve<\delta$. Then
\be
\l{cg1}
 \left(\int_0^\delta\ve^{q-sq}\|(\na F)(\cdot,\ve)\|_{L^p}^q\, \frac{d\ve}{\ve}\right)^{1/q}\lesssim |f|_{B^s_{p,q,\delta}}.
\ee
\een
\el

A word about the existence of the trace in item 1 above. We will prove below that  for every $0<\lambda<\delta/4$ we have
\be
\l{cg2}
\left|F_{|\T^n\times\{\lambda\}}\right|_{B^s_{p,q}}\lesssim \left(\int_0^{\delta/2}\ve^{q-sq}\|(\na F)(\cdot,\ve)\|_{L^p}^q\, \frac{d\ve}{\ve}\right)^{1/q}.
\ee

By Lemma \ref{ad2} and a standard argument, this leads to the existence, in $B^s_{p,q}$, of the limit $\lim_{\ve\to 0}F(\cdot,\ve)$. This limit is the trace of $F$ on $\T^n$ and clearly satisfies \eqref{ab2}.

\bp
For simplicity, we treat only the case  where $q<\infty$; the case  where  $q=\infty$ is somewhat simpler and is left to the reader.

We claim that in item 1  we may assume that $F\in C^\infty(\overline{V_\delta})$. Indeed, assume  that \eqref{ab2} holds (with $\tr F=F(\cdot, 0)$) for such $F$. By Lemma \ref{ad2}, we have the stronger inequality $\left\|\tr F-\fint\tr F\right\|_{B^s_{p,q}}\lesssim I(F)$, where $I(F)$ is the integral in \eqref{cg6}. Then, by a standard approximation argument, we find that \eqref{ab2} holds for every $F$.

So let  $F\in C^\infty(\overline{V_\delta})$, and set  $f(x):=F(x,0)$, $\fo x\in\T^n$. Denote by $I(F)$ the quantity in \eqref{cg6}. We have to prove that $f$ satisfies 
\be
\l{ab210}
|f|_{B^s_{p,q}}\lesssim I(F).
\ee 
If $|h|\le\delta$, then 
\be
\l{cg4}
|\Delta_hf(x)|\le \left|f(x+h)-F(x+h/2,|h|/2)\right|+\left|f(x)-F(x+h/2,|h|/2)\right|.
\ee
By symmetry and \eqref{cg4}, the estimate \eqref{ab210} will follow from
\be
\l{cg5}
\left(\int_{|h|\le\delta}|h|^{-sq}\|f-F(\cdot+h/2,|h|/2)\|_{L^p}^q\, \frac{dh}{|h|^n}\right)^{1/q}\lesssim I(F).
\ee
In order to prove \eqref{cg5}, we start from
\be
\l{cg8}
\begin{aligned}
\left|F(x+h/2,|h|/2)-f(x)\right|&=\left|\int_0^1 (\na F)(x+th/2,t|h|/2)\cdot (h/2,|h|/2)\, dt\right|\\
&\le |h|\int_0^1|\na F(x+th/2,t|h|/2)|\, dt.
\end{aligned}
\ee
Let $J(F)$ denote the left-hand side of \eqref{cg5}. Using \eqref{cg8} and setting $r:=|h|/2$, we obtain
\be
\l{ch1}
\begin{aligned}
[J(F)]^q&\le \int_{|h|\le\delta}|h|^{q-sq}\left(\int_0^1\|\na F(\cdot+th/2,t|h|/2)\|_{L^p}\, dt\right)^q\, \frac{dh}{|h|^n}\\
&=\int_{|h|\le\delta}|h|^{q-sq}\left(\int_0^1\|\na F(\cdot,t|h|/2)\|_{L^p}\, dt\right)^q\, \frac{dh}{|h|^n}\\
&\sim \int_0^{\delta/2}r^{q-sq-1}\left(\int_0^1\|\na F(\cdot,tr)\|_{L^p}\, dt\right)^q\, dr\\
&\sim \int_0^{\delta/2}r^{-sq-1}\left(\int_0^r\|\na F(\cdot,\sigma)\|_{L^p}\, d\sigma\right)^q\, dr
\lesssim [I(F)]^q.
\end{aligned}
\ee
The last inequality is a special case of  Hardy's inequality \cite[Chapter 5, Lemma 3.14]{steinweiss}, that we recall here when $\delta =\infty$.\footnote{ But the argument adapts to a finite $\delta$; see e.g. \cite[Proof of Corollary 7.2]{bousquetmironescu}.} Let $1\le q<\infty$ and $1<\rho<\infty$. If $G\in W^{1,1}_{loc}([0,\infty))$, then
\be
\l{e04269}
\int_0^\infty\frac{|G(r)-G(0)|^q}{r^\rho}\,dr\leq \left(\frac{q}{\rho-1}\right)^q\int_{0}^\infty\frac{|G'(r)|^q}{r^{\rho-q}}\,dr.
\ee 
We obtain \eqref{ch1} by applying \eqref{e04269} with $G'(r):=\|\na F(\cdot, r)\|_{L^p}$ and $\rho:=sq+1$.
The proof of item 1 is complete.

We next turn to item 2. We have 
\be
\l{kh2}
 \na F(x,\ve)=\frac 1\ve f\ast \eta_\ve(x),
 \ee 
 where $\na$ stands for $(\p_1,\ldots,\p_n,\p_{\ve})$. Here,    $\eta=(\eta^1,\ldots,\eta^{n+1})\in C^\infty(\T^n ; \R^{n+1})$ is supported in $\{ |x|\le 1\}$ and is given in coordinates by 
 \be
 \l{kh3}
 \eta^j=\p_j\rho, \ \fo j\in \llbracket 1, n\rrbracket,\ \eta^{n+1}=-\div (x\rho).
 \ee 
 Noting that $\int \eta=0$, we find that
\be
\l{ch2}
\begin{aligned}
|\na F(x,\ve)|&= \frac 1\ve\left|\int_{|y|\le \ve}(f(x-y)-f(x))\eta_\ve(y)\, dy\right|\\
&\lesssim\frac 1{\ve^{n+1}}\int_{|h|\le\ve}|f(x+h)-f(x)|\, dh.
\end{aligned}
\ee
Integrating \eqref{ch2}  and using Minkowski's inequality, we obtain
\be
\l{ci1}
\|\na F(\cdot,\ve)\|_{L^p}\lesssim \frac 1{\ve^{n+1}}\int_{|h|\le\ve}\|\Delta_hf\|_{L^p}\, dh.
\ee
Let $L(F)$ be the quantity in  the left-hand side of  \eqref{cg1}. Combining \eqref{ci1} with 
H\" older's inequality, we find that
\be
\l{mj1}
\begin{aligned}
 [L(F)]^q&\Bk \lesssim \int_0^{\delta}\frac 1{\ve^{nq+sq+1}}\left(\int_{|h|\le\ve}\|\Delta_hf\|_{L^p}\, dh\right)^q\, d\ve\\
&\lesssim \int_0^{\delta}\frac 1{\ve^{nq+sq+1}}\ve^{n(q-1)}\int_{|h|\le\ve}\|\Delta_hf\|_{L^p}^q\, dh\, d\ve\\
&\lesssim \int_{|h|\le\delta}|h|^{-sq}\|\Delta_hf\|_{L^p}^q\, \frac{dh}{|h|^n}=|f|_{B^s_{p,q,\delta}}^q,\end{aligned}
\ee
i.e, \eqref{cg1} holds.
\ep
 In the same vein, we have the following result, involving the semi-norm appearing in Proposition \ref{p2.4}, more specifically the quantity
  \be
\l{kd4}
|f|_{B^1_{p,q,\delta}}:=\left(\int_{|h|\le\delta}|h|^{-q}\|\Delta_h^2 f\|_{L^p}^q\, \frac{dh}{|h|^n}\right)^{1/q}
\ee
when $q<\infty$, 
with the obvious modification when $q=\infty$. 
We first introduce a notation. Given $F\in C^2(V_\delta)$, we let $D^2_\#F$ denote the collection of the second order derivatives of $F$ which are either completely horizontal (that is of the form $\p_j\p_k F$, with $j,k\in\llbracket 1,n\rrbracket$), or completely vertical (that is $\p_{n+1}\p_{n+1}F$).   
\bl
\l{kb2}
Let $1\le p<\infty$ and $1\le q\le\infty$. 
Let $F\in C^\infty(V_{\delta})$ and set 
\bes
M(F):=\left(\int_0^{\delta}\ve^{q}\|(\na F)(\cdot,\ve)\|_{L^{2p}}^{2q}\, \frac{d\ve}\ve\right)^{1/q}
\ees
and
\bes
N(F):=\left(\int_0^{\delta}\ve^{q}\left\|(D^2_\# F)(\cdot,\ve)\right\|_{L^p}^q\frac{d\ve}\ve\right)^{1/q}
\ees
(with the obvious modification when $q=\infty$).
\ben
\item
If $M(F)<\infty$ and $N(F)<\infty$,  then $F$ has a trace $f\in B^1_{p,q}(\T^n)$, satisfying
\be
\l{kf1}
\left\|f-\fint f\right\|_{L^p}\lesssim  M(F)^{\frac 12}\ee
and
\be
\l{kb3}
|f|_{B^1_{p,q,\delta}}\lesssim N(F).
\ee
\item
Conversely, let $f\in B^1_{p,q}(\T^n ; \so)$. Let $\rho\in C^\infty$ be an even mollifier supported in $\{ |x|\le 1\}$ and set $F(x,\ve):=f\ast\rho_\ve(x)$, $x\in\T^n$, $0<\ve<\delta$. Then
\be
\l{kb4}
 M(F)+N(F)\lesssim |f|_{B^1_{p,q,\delta}}.
\ee
\een
\el
The above result is inspired by the proof of \cite[ Section 10.1.1, Theorem 1, p. 512]{mazyanew}. 
The arguments we present also  lead to a (slightly different) proof of Lemma \ref{ab1}.

We start by establishing some preliminary estimates. 
We call $H\in\R^n\times\R$ \enquote{pure} if $H$ is either horizontal, or vertical, i.e., either $H\in\R^n\times\{0\}$ or $H\in \{0\}\times\R$. For further use, let us note the following fact, valid for $X\in V_\delta$ and $H\in\R^{n+1}$. 
\be
\l{ja2}H\text{ pure}\implies|D^2F(X)\cdot(H,H)|\lesssim |D^2_\#F(X)||H|^2. 
\ee
\bl
\l{jc1}
Let $X, H$ be such that $[X, X+2H]\subset \overline{V_\delta}$. Let $F\in C^2(\overline{V_\delta})$.  Then 
\be
\l{jc2}
|\Delta_H^2F(X)|\le \int_0^2 \tau |D^2F(X+\tau H)\cdot(H,H)|\, d\tau.
\ee
In particular, if $H$ is pure and we write $H=|H|K$, then 
\be
\l{jc3}
|\Delta_H^2F(X)|\lesssim \int_0^{2|H|} t |D^2_\# F(X+tK)|\, dt.
\ee
\el
\bp
Set 
\bes
G(s):=F(X+(1-s)H)+F(X+(1+s)H), \ s\in [0,1],
\ees 
so that $G\in C^2$ and in addition we have
%
\be
\l{ja3}
G'(0)=0,\ G''(s)=[D^2F(X+(1-s)H)+D^2F(X+(1+s)H)]\cdot(H,H),
\ee 
and
\be
\l{ja4}
\int_0^1(1-s)G''(s)\, ds=G(1)-G(0)-G'(0)=\Delta_H^2F(X).
\ee
Estimate \eqref{jc2} is a consequence of \eqref{ja3} and \eqref{ja4} (using the changes of variable $\tau:=1\pm s$). In the special case where $H$ is pure, we rely on \eqref{ja2} and \eqref{jc2} and obtain \eqref{jc3} via the change of variable  $t:=\tau|H|$. 
\ep
If we combine \eqref{jc3} (applied first with $H=(h,0)$, $h\in\R^n$, next with $H=(0,t)$, $t\in [0,\delta/2]$) with Minkowski's inequality, we obtain the two following consequences\footnote{ In \eqref{jb1}, we let $\Delta_h^2F(\cdot,\ve):=F(\cdot+2h,\ve)-2F(\cdot+h,\ve)+F(\cdot,\ve)$.}
\be
\l{jb1}
[h\in\R^n,\ 0\le\ve\le\delta]\implies \|\Delta_h^2F(\cdot,\ve)\|_{L^p}\lesssim |h|^2\|D^2_\#F(\cdot,\ve)\|_{L^p},
\ee
and\footnote{ With the slight abuse of notation $\Delta_{te_{n+1}}^2F(\cdot,\ve):=F(\cdot,\ve+2t)-2F(\cdot,\ve+t)+F(\cdot,\ve)$.}
\be
\l{jb2}
\begin{aligned}
[t, \ve\ge 0,\ \ve+2t\le\delta]
\implies \|\Delta_{te_{n+1}}^2F(\cdot,\ve)\|_{L^p}&\lesssim \int_0^{2t} r \|D^2_\#F(\cdot,\ve+r)\|_{L^p}\, dr.
\end{aligned}
\ee
\bp[Proof of Lemma \ref{kb2}]
We start by proving \eqref{kf1}. By Lemma \ref{ab1} (applied with $s=1/2$ and with $2p$ (respectively $2q$) instead of $p$ (respectively $q$)), $F$ has, on $\T^n$, a trace $\tr F\in  B^{1/2}_{2p,2q}$. By  Lemma \ref{ab1}, item $1$, and Lemma \ref{ad3},   we have 
\bes  
\left\|\tr F-\fint \tr F\right\|_{L^p}\lesssim\left\|\tr F-\fint \tr F\right\|_{L^{2p}}\lesssim   M(F)^{1/2}
\ees
i.e., \eqref{kf1} holds.

We next establish \eqref{kb3}. Arguing 
as at the beginning of the proof of Lemma \ref{ab1},  one concludes  that it suffices to prove \eqref{kb3} when $F\in C^\infty(\overline{V_\delta})$. 

So let us consider some $F\in C^\infty(\overline{V_\delta})$. We set $f(x)=F(x,0)$, $\fo x\in\T^n$. 
Then \eqref{kb3} is equivalent to 
\be
\l{kf2}
|f|_{B^1_{p,q,\delta}}\lesssim N(F).
\ee
We treat only the case where $q<\infty$; the case where $q=\infty$ is slightly simpler and is left to the reader. 

The starting point is the following identity, valid when $|h|\le\delta$ and with $t:=|h|$
\be
\l{jd1}
\begin{aligned}
\Delta_h^2f=&\Delta_{te_{n+1}/2}^2F(\cdot+2h,0)-2\Delta_{te_{n+1}/2}^2F(\cdot+h,0)+\Delta_{te_{n+1}/2}^2F(\cdot,0)\\
&+2\Delta_h^2F(\cdot, t/2)-\Delta_h^2F(\cdot, t).
\end{aligned}
\ee
By \eqref{jb1}, \eqref{jb2} and \eqref{jd1}, we find that
\be
\l{jd2}
\begin{aligned}
\|\Delta_h^2f\|_{L^p}\lesssim & \int_0^{|h|}r\|D^2_\#F(\cdot, r)\|_{L^p}\, dr+|h|^2\|D^2_\#F(\cdot, |h|/2)\|_{L^p}\\
&+|h|^2\|D^2_\#F(\cdot, |h|)\|_{L^p}.
\end{aligned}
\ee
Finally, \eqref{jd2} combined with Hardy's inequality \eqref{e04269} (applied to the integral $\int_0^\delta$ and  with $G'(r):=r\|D^2_\#F(\cdot, r)\|_{L^p}$ and $\rho:=q+1$) yields
%
\be
\l{kf6}
\begin{aligned}
|f|_{B^1_{p,q,\delta}}^q&\lesssim \int_{|h|\le\delta}\frac 1{|h|^q}\left(\int_0^{|h|}r \left\|
D^2_\#F(\cdot, r)\right\|_{L^p}\, dr\right)^q\, \frac{dh}{|h|^n}+[N(F)]^q\\
&\lesssim [N(F)]^q.
\end{aligned}
\ee
This implies \eqref{kf2} and completes the proof of item 1.

We now turn to item 2. We claim that 
\be
\l{kg1}
|f|_{B^{1/2}_{2p,2q,\delta}}\lesssim |f|_{B^1_{p,q,\delta}}^{1/2}.
\ee
Indeed, it suffices to note the fact that 
$ |\Delta_h^2f|^{2p}\lesssim  |\Delta_h^2f|^p$ (since $|f|=1$). 
By combining \eqref{kg1} with 
Lemma \ref{ab1}, we find that
\be
\l{kg2}
M(F)=\left(\int_0^{\delta}\ve^{q}\|(\na F)(\cdot,\ve)\|_{L^{2p}}^{2q}\, \frac{d\ve}\ve\right)^{1/q}\lesssim |f|_{B^1_{p,q,\delta}}.
\ee
Thus, in order to complete the proof of \eqref{kb4}, it suffices to combine \eqref{kg2} with the following estimate
\be
\l{kg3}
N(F)
\lesssim |f|_{B^1_{p,q,\delta}},
\ee
that we now establish. The key argument for proving \eqref{kg3} is the following second order analog of \eqref{ch2}:
\be
\l{ki1}
|D^2_\#F(x,\ve)|\lesssim \frac 1{\ve^{n+2}}\int_{|h|\le\ve}|\Delta_h^2f(x-h)|\, dh.
\ee
The proof of \eqref{ki1} appears in \cite[p. 514]{mazyanew}. For the sake of completeness, we reproduce below the argument. First, 
differentiating the expression defining $F$,
we have
\be
\l{ki2}
\p_j\p_kF(x,\ve)=\frac 1{\ve^2}f\ast(\p_j\p_k\rho)_\ve,\ \fo j,\, k\in\llbracket 1,n\rrbracket.
\ee
Using \eqref{ki2} and the fact that $\p_j\p_k\rho$ is even and has zero average, we obtain the identity
\bes
\p_j\p_kF(x,\ve)=\frac 1{2\ve^{n+2}}\int_{|h|\le\ve}\p_j\p_k\rho(h/\ve)\Delta_h^2f(x-h)\, dh,
\ees
and thus \eqref{ki1} holds for the derivatives $\p_j\p_kF$, with  $j,\, k\in\llbracket 1,n\rrbracket$.

We next note the identity
\be
\l{ki4}
F(x,\ve)=\frac 1{2\ve^n}\int \rho(h/\ve)\Delta_h^2f(x-h)\, dh+f(x),
\ee
which follows from the fact that $\rho$ is even.

By differentiating twice \eqref{ki4} with respect to $\ve$, we obtain that \eqref{ki1} holds when $j=k=n+1$. The proof of \eqref{ki1} is complete.

Using \eqref{ki1} and Minkowski's inequality, we obtain
\be
\l{mi1}
\|D^2_\#F(\cdot,\ve)\|_{L^p}\lesssim \frac 1{\ve^{n+2}}\int_{|h|\le\ve}\|\Delta_h^2f\|_{L^p}\, dh,
\ee
which is a second order analog of \eqref{ci1}. Once \eqref{ci1} is obtained, we repeat the calculation leading to 
\eqref{mj1} and obtain \eqref{kg3}. The details are left to the reader.

The proof of Lemma \ref{kb2} is complete.
\ep
\br
\l{av1}
One may put Lemmas \ref{ab1} and \ref{kb2}  in the perspective of the theory of weighted Sobolev spaces. Let us start by recalling one of the striking achievements of this theory. As it is well-known, we have $\tr W^{1,1}(\R^n_+)=L^1(\R^{n-1})$, and, when $n\ge 2$, the trace operator has no linear continuous right-inverse $T:L^1(\R^{n-1})\to W^{1,1}(\R^n)$ \cite{gagliardo}, \cite{peetre}. The expected analogs of these facts for $W^{2,1}(\R^n_+)$ are both wrong. More specifically, we have $\tr W^{2,1}(\R^n_+)=B^1_{1,1}(\R^{n-1})$ (which is a strict subspace of $W^{1,1}(\R^{n-1}))$, and the trace operator has a linear continuous right inverse from $B^1_{1,1}(\R^{n-1})$ into $W^{2,1}(\R^n_+)$. These results are special cases of the trace theory for weighted Sobolev spaces developed by Uspenski\u\i {} \cite{uspenskii}. For a modern treatment of this theory, see e.g. \cite{tracesoldnew}.
\er
%
\subsection{Product estimates}
${}$

Lemma \ref{at3} below is a variant of \cite[Lemma D.2]{lss}. Here, $\Omega$ is either smooth bounded, or $(0, 1)^n$, or $\T^n$.

\bl
\l{at3} Let $s>1$, $1\le p<\infty$ and $1\le q\le \infty$. If $u, v\in B^s_{p,q}\cap L^\infty(\Omega)$, then $u\na v\in B^{s-1}_{p,q}$.
\el

\bp
After extension to $\R^n$ and cutoff, 
 we may assume that $u, v\in B^s_{p,q}\cap L^\infty$. It thus suffices to prove that $u, v\in B^s_{p,q}\cap L^\infty(\R^n)$$\implies$$u\na v\in B^{s-1}_{p,q}(\R^n)$. 

In order to prove the above, we argue as follows. Let $u=\sum u_j$ and $v=\sum v_j$ be the Littlewood-Paley decompositions of $u$ and $v$. Set 
\bes
f^j:=\sum_{k\le j}u_k\na v_j+\sum_{k<j}u_j\na v_k.
\ees 
Since $\supp{\mathcal F} (u_k\na v_j)\subset B(0, 2^{\max\{ k, j\}+2})$, we find that $u\na v=\sum f^j$ is a Nikolski\u\i {} decomposition of $u\na v$; see Section \ref{mm8}. Assume e.g. that $q<\infty$. In view of Proposition \ref{mm9}, the conclusion of Lemma \ref{at3} follows if we prove that
\be
\l{mn1}
\sum 2^{(s-1)jq}\|f^j\|_{L^p}^q<\infty.
\ee
In order to prove \eqref{mn1}, we rely on the elementary estimates \cite[Lemma 2.1.1, p. 16]{chemin}, \cite[formulas (D.8), (D.9), p. 71]{lss} 
\be
\l{mn2}
\left\|\sum_{k\le j}u_k\right\|_{L^\infty}\lesssim \|u\|_{L^\infty}, \quad \fo j\ge 0,
\ee
\be
\l{mn3}
\left\| \sum_{k< j}\na v_k\right\|_{L^\infty}\lesssim 2^j\|v\|_{L^\infty}, \quad \fo j\ge 0,
\ee
and
\be
\l{mn4}
\|\na v_j\|_{L^p}\lesssim 2^j\|v_j\|_{L^p},\quad \fo j\ge 0.
\ee
By combining \eqref{mn2}-\eqref{mn4}, we obtain
\bes
\begin{aligned}
\sum 2^{(s-1)jq}\|f^j\|_{L^p}^q&\lesssim \sum 2^{(s-1)jq}\left(\left\|\sum_{k\le j}u_k\right\|_{L^\infty}^q\|\na v_j\|_{L^p}^q+\left\|\sum_{k<j}\na v_k\right\|_{L^\infty}^q\| u_j\|_{L^p}^q\right)\\
&\lesssim \|u\|_{L^\infty}^q\sum 2^{sjq}\|v_j\|_{L^p}^q+\|v\|_{L^\infty}^q\sum 2^{sjq}\|u_j\|_{L^p}^q
\\
&\lesssim \|u\|_{L^\infty}^q\|v\|_{B^s_{p,q}}^q+\|v\|_{L^\infty}^q\|u\|_{B^s_{p,q}}^q,
\end{aligned}
\ees
and thus \eqref{mn1} holds.
\ep

\subsection{Superposition   operators}
${}$

In this section, we examine the mapping properties of the operator
\bes
T_\Phi,\  \psi\xmapsto{T_\Phi} \Phi\circ\psi.
\ees

We work in $\Omega$ smooth bounded, or $(0,1)^n$, or $\T^n$.

The next result is classical and straightforward; see e.g. \cite[Section 5.3.6, Theorem 1]{runstsickel}.
 \begin{lemma} \label{eipsi}
 Let $0<s<1$,  $1\le p<\infty$, and $1\le q<\infty$. Let $\Phi:\R^k\to\R^l$ be a Lipschitz function
 . 
Then  $
T_\Phi$    
maps $B^{s}_{p,q}(\Omega ;  \R^k)$ into $B^{s}_{p,q}(\Omega ;  \R^l)$.

Special case:   $\psi\mapsto e^{{\im} \psi}$  maps $B^s_{p,q}(\Omega ; \R)$ into $B^s_{p,q}(\Omega ; \so)$.

In addition, when $q<\infty$, $T_\Phi$ is continuous. 
 \end{lemma}

For the next result, see \cite[Section 5.3.4, Theorem 2, p. 325]{runstsickel}. 
\bl
\l{ka2}
Let $s>0$, $1\le p<\infty$ and $1\le q\le\infty$. Let $\Phi\in C^\infty(\R^k ; \R^l)$. Then $T_\Phi$  maps $(B^{s}_{p,q}\cap L^\infty)(\Omega ;  \R^k)$ into $(B^{s}_{p,q}\cap L^\infty)(\Omega ;  \R^l)$.

Special case:   $\psi\mapsto e^{{\im} \psi}$ maps $(B^{s}_{p,q}\cap L^\infty)(\Omega ; \R)$ into $(B^{s}_{p,q}\cap L^\infty)(\Omega ; \so)$.
\el

\subsection{Integer valued functions}
${}$

The next result is a cousin of  \cite[Appendix B]{lss},\footnote{ The context there is the one of the Sobolev spaces.} but the argument in \cite{lss} does not seem to apply in our situation. Lemma \ref{Eunicite} can be obtained from the results in \cite{bbmuni}, but we present below a simpler direct argument. 
 \begin{lemma}\label{Eunicite}Let $s>0$, $1\le p<\infty$ and $1\le q<\infty$ be such that $sp\ge 1$. Then the functions in $B^{s}_{p,q}(\Omega ;\mathbb{Z})$ are constant. 
 
Same result when $s>0$, $1\le p<\infty$, $q=\infty$ and $sp>1$.
 
 The same conclusion holds for functions in $\sum_{j=1}^k B^{s_j}_{p_j,q_j}(\Omega ; \mathbb{Z})$, provided we have for all  $j\in\llbracket 1,k\rrbracket$:  either $s_jp_j=1$ and $1\le q_j<\infty$, or $s_jp_j>1$ and $1\le q_j\le\infty$.
 \end{lemma}
\begin{proof}
The case where  $n=1$ is simple. Indeed, by Lemma \ref{B-VMO}  we have $B^{s}_{p,q}\hookrightarrow \VMO$  (and similarly $\sum_{j=1}^k B^{s_j}_{p_j,q_j}\hookrightarrow \VMO$). The conclusion follows from the fact that $\VMO((0,1) ;\Z)$ functions are constant \cite[Step 5, p. 229]{brezisnirenberg1}. 

 We next turn to the general case. Let $f=\sum_{j=1}^kf_j$, with $f_j\in B^{s_j}_{p_j,q_j}(\Omega ; \Z)$, $\fo j\in\llbracket 1, k\rrbracket$. In view of the conclusion, we may assume that $\Omega =(0,1)^n$. By the Sobolev embeddings, we may assume that for all $j$ we have $s_jp_j=1$  (and thus  either $1<p_j<\infty$ and $s_j=1/p_j$, or $p_j=1$ and $s_j=1$) and $1\le q_j<\infty$. Let, as in Lemma \ref{ad1}, $A\subset (0,1)^{n-1}$ be a set of full measure such that \eqref{cf1} holds with $M=2$. The  proof of the lemma relies on the following key implication:
 \be
 \l{cf2}
 [x_1+\cdots+x_k\in\Z,\ 1\le p_1,\ldots, p_k<\infty]\implies |x_1+\cdots+x_k|\lesssim |x_1|^{p_1}+\cdots+|x_k|^{p_k}.
 \ee
 This leads to the following consequence: if $g:=g_1+\cdots+g_k$ is integer-valued, then 
 \be
 \l{aaa1}
 \|\Delta_h^2g\|_{L^1}\lesssim \|\Delta_h^2g_1\|_{L^{p_1}}^{p_1}+\cdots+\|\Delta_h^2g_k\|_{L^{p_k}}^{p_k}.
\ee
By combining \eqref{cf1} with \eqref{aaa1},
 we find that
\be
\l{cf3}
\lim_{l\to\infty}\frac{\left\|\Delta_{t_le_n}^2f(x',\cdot)\right\|_{L^{1}((0,1))}}{t_l}=0,\quad\fo x'\in A,\text{ for some sequence }t_l\to 0.
\ee
 By Lemma \ref{cf4} below, we find that $f(x',\cdot)$ is constant, for every $x'\in A$. By a permutation of the coordinates, we find that
for every $i\in \llbracket 1, n\rrbracket$, the function 
\be
\l{aa2}
\text{$t\mapsto f(x_1,...,x_{i-1},t, x_{i+1},...,x_n)$ is constant, $\fo  i\in \llbracket 1, n\rrbracket$,   a.e. $\widehat x_i \in (0,1)^{n-1}$;} 
\ee
here, $\widehat x_i:=(x_1,...,x_{i-1},x_{i+1},...,x_n) \in (0,1)^{n-1}$.

 We next invoke the fact that every measurable function satisfying \eqref{aa2} is constant  \cite[Lemma 2]{blmn}.            
\end{proof} 
\bl
\l{cf4}
Let $g\in L^1((0,1) ; \Z)$ be such that, for some sequence $t_l\to 0$, we have
\be
\l{cf5}
\lim_{l\to\infty}\frac{\left\|\Delta_{t_l}^2g\right\|_{L^{1}((0,1))}}{t_l}=0.
\ee
Then $g$ is constant.
\el
\bp
In order to explain the main idea, let us first assume that $g=\one_B$ for some $B\subset (0,1)$. Let $h\in (0,1)$. If $x\in B$ and $x+2h\not\in B$, then $\Delta_h^2g(x)$ is odd, and thus $|\Delta_h^2g(x)|\ge 1$. The same holds if $x\not\in B$ and $x+2h\in B$. On the other hand, we have $|\Delta_{2h}g(x)|\le 1$, with equality only when either $x\in B$ and $x+2h\not\in B$, or $x\not\in B$ and $x+2h\in B$. By the preceding, we obtain the inequality 
\be
\l{cf6}|\Delta^2_hg(x)|\ge |\Delta_{2h}g(x)|,\quad\fo x,\, \fo h.
\ee 
Using \eqref{cf5} and \eqref{cf6}, we obtain
\be
\l{cf7}
g'=\lim_{l\to\infty}\frac{\Delta_{2t_l}g}{2t_l}=0.\footnotemark
\ee
\footnotetext{
  In \eqref{cf7}, the first limit is in ${\cal D}'$, the second one in $L^1$.} 
Thus either $g=0$, or $g=1$.

We next turn to the general case. Consider some $k\in\Z$ such that the measure of the set $g^{-1}(\{k\})$ is positive. We may assume that $k=0$, and we will prove that $g=0$. For this purpose, we set $B:=g^{-1}(2\Z)$, and we let $\overline g:=\one_B$. Arguing as above, we have 
$ |\Delta^2_h g(x)|\ge |\Delta_{2h}\overline g(x)|$, $\fo x$, $\fo h$,
and thus $\overline g=0$. We find that $g$ takes only even values. We next consider the integer-valued map $g/2$. By the above, $g/2$ takes only even values, and so on. We find that $g=0$.
\ep

\subsection{Disintegration of the Jacobians}
\l{au1}
${}$

The purpose of this section is to prove and generalize the following result, used in the analysis of Case \ref{Y}.
\bl
\l{at6}
Let $s>1$, $1\le p<\infty$, $1\le q\le p$ and $n\ge 3$, and assume   that $sp\ge 2$. Let $u\in B^s_{p,q}(\Omega ; \so)$ and set $F:=u\wedge\na u$. 
Then $\curl F=0$.

Same conclusion if $s>1$, $1\le p<\infty$, $1\le q\le \infty$ and $n\ge 2$,  and we have $sp>2$.

Same conclusion if $s>1$, $1\le p<\infty$, $1\le q< \infty$ and $n=2$,  and we have $sp=2$.
\el

In view of the conclusion, we may assume that $\Omega=(0,1)^n$.

Note that in the above we have $n\ge 2$; for $n=1$ there is nothing to prove. 

Since the results we present in this section are of independent interest, we go beyond what is actually needed in Case \ref{Y}.

The conclusion of (the generalization of) Lemma \ref{at6} relies on three ingredients. The first one is that it is possible to define, as a distribution, the product $F: =u\wedge\na u$ for $u$ in a low regularity Besov space; this goes back to \cite{lddjr} when $n=2$, and the case where $n\ge 3$ is treated in \cite{bousquetmironescu}.  The second one is a Fubini (disintegration) type result for the distribution $\curl F$. Again, this result holds even in Besov spaces with lower regularity than the ones in Lemma \ref{at6}; see Lemma \ref{mo2} below. The final ingredient is  the fact that when  $u\in\VMO((0,1)^2 ; \so)$ we have $\curl F=0$; see Lemma \ref{mo3}. Lemma \ref{at6} is obtained by combining Lemmas \ref{mo2} and \ref{mo3} via a dimensional reduction (slicing) based on Lemma \ref{mo7}; a more general result is presented in Lemma \ref{mo4}.

Now let us proceed. First, following \cite{lddjr} and \cite{bousquetmironescu}, we explain how to define the Jacobian $Ju:=1/2\curl F$ of low regularity unimodular maps $u\in W^{1/p,p}((0,1)^n ; \so)$, with $1\le p<\infty$.\footnote{ In \cite{lddjr} and \cite{bousquetmironescu}, maps are from ${\mathbb S}^n$ (instead of $(0,1)^n$) into $\so$, but this is not relevant for the validity of the results we present here.} Assume first that $n=2$ and that $u$ is smooth. Then, in the distributions sense, we have
\be
\l{oa2}
\begin{aligned}
\la Ju,\zeta\ra&=\frac 12\int_{(0,1)^2}\curl F\, \zeta=-\frac 12\int_{(0,1)^2}\na\zeta\wedge(u\wedge\na u)\\&=\frac 12\int_{(0,1)^2}[(u\wedge\p_1u)\p_2\zeta-(u\wedge\p_2u)\p_1\zeta]\\
&=\frac 12\int_{(0,1)^2}(u_1\na u_2\wedge\na\zeta-u_2\na u_1\wedge\na\zeta),\quad\fo \zeta\in C^\infty_c((0,1)^2). 
\end{aligned}
\ee
In higher dimensions, it is better to identify $Ju$ with the $2$-form (or rather a $2$-current) $Ju\equiv 1/2\,  d(u\wedge du)$.\footnote{ We recover the two-dimensional formula \eqref{oa2} via the usual identification of $2$-forms on $(0,1)^2$ with scalar functions (with the help of the Hodge $\ast$-operator).} With this identification and modulo the action of the Hodge $\ast$-operator, $Ju$ acts  
either or $(n-2)$-forms, or on $2$-forms. The former point of view is usually adopted, and is expressed by the formula
\be
\l{oa3}
\begin{aligned}
\la Ju,\zeta\ra&=\frac {(-1)^{n-1}}2\int_{(0,1)^n}d\zeta\wedge(u\wedge\na u)\\
&=\frac {(-1)^{n-1}}2\int_{(0,1)^n}d\zeta\wedge(u_1\, du_2-u_2\, du_1),\quad\fo \zeta\in C^\infty_c(\Lambda^{n-2}(0,1)^n).
\footnotemark
\end{aligned}
\ee
\footnotetext{ Here, $C^\infty_c(\Lambda^{n-2}(0,1)^n)$ denotes the space of smooth compactly supported $(n-2)$-forms on $(0,1)^n$.}
The starting point in   extending the above formula to lower regularity maps $u$ is provided by the identity \eqref{oa4} below; when  $u$ is smooth, \eqref{oa4}   is obtained  by a simple integration by parts.
More specifically, consider any smooth extension  $U:(0,1)^n\times [0,\infty)\to\C$, respectively $\varsigma\in C^\infty_c(\Lambda^{n-2}((0,1)^n\times [0,\infty)))$ of $u$, respectively of $\zeta$.\footnote{ We do not claim that $U$ is $\so$-valued. When $u$ is not smooth, existence of $\so$-valued extensions is a delicate matter \cite{soreview}.} Then we have the identity \cite[Lemma 5.5]{bousquetmironescu}
\be
\l{oa4}
\la Ju,\zeta\ra=(-1)^{n-1}\int_{(0,1)^n\times (0,\infty)}d\varsigma\wedge\ dU_1\wedge dU_2.
\ee
For a low regularity $u$ and for a well-chosen $U$, we take the right-hand side of \eqref{oa4} as the definition of $Ju$. More specifically, let $\Phi\in C^\infty(\R^2; \R^2)$ be such that $\Phi(z)=z/|z|$ when $|z|\ge 1/2$, and let $v$ be a standard extension of $u$ by averages, i.e., $v(x,\ve)=u\ast\rho_\ve(x)$, $x\in (0,1)^n$, $\ve>0$, with $\rho$ a standard mollifier. Set $U:=\Phi(v)$. With this choice of $U$, the right-hand side of \eqref{oa4} does not depend on $\varsigma$ (once $\zeta$ is fixed) \cite[Lemma 5.4]{bousquetmironescu} and the map $u\mapsto Ju$ is continuous from $W^{1/p,p}((0,1)^n ; \so)$ into the set of $2$- (or $(n-2)$-)currents. When $p=1$, continuity is straightforward. For the continuity when $p>1$, see \cite[Theorem 1.1 item 2]{bousquetmironescu}. In addition, when $u$ is sufficiently smooth (for example when $u\in W^{1,1}((0,1)^n ; \so)$), $Ju$ coincides\footnote{ Up to the action of the $\ast$ operator.} with $\curl F$ \cite[Theorem 1.1 item 1]{bousquetmironescu}. Finally, we have the estimate \cite[Theorem 1.1 item 3]{bousquetmironescu}
\be
\l{oa6}
|\la Ju,\zeta\ra|\lesssim |u|_{W^{1/p,p}}^p\|d\zeta\|_{L^\infty},\quad\fo \zeta\in C^\infty_c(\Lambda^{n-2}(0,1)^n).
\ee

We are now in position to explain disintegration along two-planes. We use the notation in Section \ref{mo6}. Let $u\in W^{1/p,p}((0,1)^n ; \so)$, with $n\ge 3$. Let $\alpha\in I(n-2, n)$. Then for a.e. $x_\alpha\in (0,1)^{n-2}$, the partial map $u_\alpha(x_\alpha)$ belongs to $W^{1/p,p}((0,1)^2 ; \so)$ (Lemma \ref{oa1}), and therefore $Ju_\alpha(x_\alpha)$ makes sense and acts on functions.\footnote{ Or rather on $2$-forms, in order to be consistent with our construction in dimension $\ge 3$.} Let now $\zeta\in C^\infty_c(\Lambda^{n-2}(0,1)^n)$. Then we may write
\bes
\zeta=\sum_{\alpha\in I(n-2,n)}\zeta^\alpha\, dx^{\alpha}=\sum_{\alpha\in I(n-2,n)}\left(\zeta^\alpha\right)_\alpha(x_{\overline\alpha})\, dx^{\alpha}.
\ees
Here,  $dx^{\alpha}$ is the canonical $(n-2)$-form induced by the coordinates $x_j$, $j\in\alpha$, and  $(\zeta^\alpha)_\alpha(x_{\overline\alpha})=\zeta^\alpha(x_\alpha, x_{\overline\alpha})$ belongs to $C^\infty_c((0,1)^2)$ (for fixed $x_\alpha$).

We next note the following formal calculation. 
Fix $\alpha\in I(n-2,n)$, and let $\overline\alpha=\{ j, k\}$, with $j<k$. Then 
\bes
\begin{aligned}
2(-1)^{n-1}\la Ju,\zeta^\alpha\, dx^\alpha\ra&=\int_{(0,1)^n}d(\zeta^\alpha\, dx^\alpha)\wedge(u\wedge \na u)\\
&=\int_{(0,1)^n}(\p_j\zeta^\alpha\,  dx_j+\p_k\zeta^\alpha\, dx_k)\wedge dx^\alpha\wedge u\wedge (\p_j u\, dx_j+\p_k u\, dx_k)\\
&=\int_{(0,1)^n}(\p_j\zeta^\alpha\, u\wedge \p_k u-\p_k\zeta^\alpha\,  u\wedge \p_j u)\, dx_j\wedge dx^\alpha\wedge dx_k,
\end{aligned}
\ees
that is,
\be
\l{oc2}
\la Ju,\zeta\ra=\frac 12\ \sum_{\alpha\in I(n-2,n)}\ve(\alpha)\int_{(0,1)^{n-2}}\la Ju_\alpha,\left(\zeta^\alpha\right)_\alpha(x_\alpha)\ra\, dx_\alpha, 
\ee
where $\ve(\alpha)\in \{-1, 1\}$ depends on $\alpha$. 

When $u\in W^{1,1}((0,1)^n ; \so)$,  it is easy to see  that \eqref{oc2} is true (by Fubini's theorem). The validity of  \eqref{oc2} under weaker regularity assumptions is the content of our next result.
 \bl
\l{mo2}
Let $1\le p<\infty$ and $n\ge 3$. Let $u\in W^{1/p,p}((0,1)^n ; \so)$. Then \eqref{oc2} holds.
\el
\bp
The case $p=1$ being clear, we may assume that $1<p<\infty$. We may also assume that $\zeta=\zeta^\alpha\, dx^\alpha$ for some fixed $\alpha\in I(n-2,n)$. A first ingredient of the proof of \eqref{oc2} is  the density of $W^{1,1}((0,1)^n ; \so)\cap W^{1/p,p}((0,1)^n ; \so)$ into $W^{1/p,p}((0,1)^n ; \so)$ \cite[Lemma 23]{bbmihes}, \cite[Lemma A.1]{lddjr}. Next, we note that the left-hand side of \eqref{oc2} is continuous with respect to the $W^{1/p,p}$ convergence  of unimodular maps \cite[Theorem 1.1 item 2]{bousquetmironescu}. In addition, as we noted, \eqref{oc2} holds when $u\in W^{1,1}((0,1)^n ; \so)$. Therefore, it suffices to prove that the right-hand side of \eqref{oc2} is continuous with respect to $W^{1/p,p}$ convergence of $\so$-valued maps. This is proved as follows. Let $u_j, u\in W^{1/p,p}((0,1)^n ; \so)$ be such $u_j\to u$ in $W^{1/p,p}$. By a standard argument,  since the right-hand side of \eqref{oc2} is uniformly bounded with respect to $j$ by \eqref{oa6},  it suffices to prove that the right-hand side of \eqref{oc2} corresponding to $u_j$ tends to the one corresponding to $u$ possibly along a subsequence. 

In turn, convergence up to a subsequence is proved as follows. Recall the following vector-valued version  of the \enquote{converse} to the dominated convergence theorem \cite[Theorem 4.9, p. 94]{brezisfa}.  If $X$ is a Banach space, $\omega$ a measured space and $f_j\to f$ in $L^p(\omega, X)$, then (possibly along a subsequence)  for a.e. $\varpi\in\omega$ we have $f_j(\varpi,\cdot )\to f(\varpi,\cdot)$ in $X$, and in addition there exists some $g\in L^p(\omega)$ such that $\|f_j(\varpi,\cdot)\|_{X}\le g(\varpi)$ for a.e. $\varpi\in \omega$. 

Using the above and Lemma \ref{oa1} item 2 (applied with $s=1/p$), we find that, up to a subsequence, we have 
\be
\l{oc3}
(u_j)_\alpha(x_\alpha)\to u_\alpha(x_\alpha)\text{ in }W^{1/p,p}((0,1)^2 ; \so)\text{ for a.e. }x_\alpha\in (0,1)^{n-2},
\ee 
and in addition we have, for some  $g\in L^p((0,1)^{n-2})$,
\be
\l{oc4}
|(u_j)_\alpha(x_\alpha)|_{W^{1/p,p}((0,1)^2)}\le g(x_\alpha)\text{ for a.e. }x_\alpha\in(0,1)^{n-2}.
\ee
The continuity of the right-hand side of \eqref{oc2} (along some subsequence) is obtained by combining \eqref{oc3} and \eqref{oc4} with \eqref{oa6} (applied with $n=2$).\footnote{In order to be complete, we should also check that the right-hand side of \eqref{oc2} is measurable with respect to $x_\alpha$. This is clear when $u\in W^{1,1}((0,1)^n ; \so)$. The general case follows by density and \eqref{oc3}.}  
\ep
\bl
\l{mo3}
Let $1\le p<\infty$. Let $u\in W^{1/p,p}\cap\VMO((0,1)^2 ; \so)$. Then $Ju=0$.
\el
\bp
Assume first that in addition we have $u\in C^\infty$. Then $u=e^{\im\va}$ for some $\va\in C^\infty$, and thus $Ju=1/2\curl (u\wedge\na u)=1/2\curl\na\va=0$. 

We now turn to the general case. Let $F(x,\ve):=u\ast\rho_\ve(x)$, with $\rho$ a standard mollifier. Since $u\in\VMO((0,1)^2 ; \so)$, there exists some  $\delta>0$ such that $1/2<|F(x,\ve)|\le 1$ when $0<\ve<\delta$ (see \eqref{boundsv} and the discussion in Case \ref{X}). Let $\Phi\in C^\infty(\R^2 ; \R^2)$ be such that $\Phi(z):=z/|z|$ when $|z|\ge 1/2$,  and define $F_\ve(x):=F(x,\ve)$ and $u_\ve:=\Phi\circ F_\ve$, $\fo 0<\ve<\delta$. Then $F_\ve\to u$ in $W^{1/p,p}$ and (by Lemma \ref{eipsi} when $p>1$, respectively by a straightforward argument when $p=1$) we have $u_\ve=\Phi(F_\ve)\to \Phi(u)=u$  in $W^{1/p,p}((0,1)^2 ; \so)$ as $\ve\to 0$. Since (by the beginning of the proof) we have $Ju_\ve=0$, we conclude via the continuity of $J$ in $W^{1/p,p}((0,1)^2 ; \so)$ \cite[Theorem 1.1 item 2]{bousquetmironescu}. 
\ep
We may now state and prove  the following generalization of Lemma \ref{at6}.
\bl
\l{mo4}
Let $s>0$, $1\le p<\infty$, $1\le q\le  p$, $n\ge 3$, and assume that  $sp\ge 2$. Let $u\in B^s_{p,q}(\Omega ; \so)$. Then $Ju=0$. 

Same conclusion if $s>0$, $1\le p<\infty$, $1\le q\le\infty$, $n\ge 2$,   and we have $sp>2$.

Same conclusion if $s>0$, $1\le p<\infty$,  $1\le q<\infty$, $n= 2$,   and we have $sp=2$.
\el
\bp
We may assume that $\Omega=(0,1)^n$. 
By the Sobolev embeddings (Lemma \ref{Besovemb}), it suffices to consider the limiting case where:
\ben
\item
$s>0$, $1\le p<\infty$, $1\le q<\infty$, $n=2$, and $sp=2$.

Or
\item
$s>0$, $1\le p<\infty$, $q= p$, $n\ge 3$, and $sp=2$.
\een
In view of Lemmas \ref{Besovemb} and \ref{B-VMO}, the case where $n=2$ is covered by Lemma \ref{mo3}. Assume that $n\ge 3$. Then the desired conclusion is obtained by combining Lemmas \ref{oa1}, \ref{mo7}, \ref{mo2} and \ref{mo3}.
\ep
\br
\l{oc1}
Arguments similar to the one developed in this section lead to the conclusion that the Jacobians of maps $u\in W^{s,p}((0,1)^n ;  {\mathbb S}^k)$, defined when $sp\ge k$ \cite{lddjr}, \cite{bousquetmironescu},  disintegrate over $(k+1)$-planes.
When $s=1$ and $p\ge k$, this assertion is implicit in \cite[Proof of Proposition 2.2, pp. 701-704]{isobe2}.
\er


\begin{thebibliography}{10}
\bibitem{bethuelchiron}
F.~Bethuel and D.~Chiron.
\newblock Some questions related to the lifting problem in {S}obolev spaces.
\newblock In {\em Perspectives in nonlinear partial differential equations},
  volume 446 of {\em Contemp. Math.}, pages 125--152. Amer. Math. Soc.,
  Providence, RI, 2007.

\bibitem{bethuelzheng}
F.~Bethuel and X.M. Zheng.
\newblock Density of smooth functions between two manifolds in {S}obolev
  spaces.
\newblock {\em J. Funct. Anal.}, 80(1):60--75, 1988.

\bibitem{bourdaud}
G.~Bourdaud.
\newblock {Ondelettes et espaces de {B}esov}.
\newblock {\em {Rev. Mat. Iberoamericana}}, 11(3):477--512, 1995.

\bibitem{lss}
J.~Bourgain, H.~Brezis, and P.~Mironescu.
\newblock {Lifting in Sobolev spaces}.
\newblock {\em J. Anal. Math.}, 80:37--86, 2000.

\bibitem{leta}
J.~Bourgain, H.~Brezis, and P.~Mironescu.
\newblock Limiting embedding theorems for {$W\sp {s,p}$} when {$s\uparrow1$}
  and applications.
\newblock {\em J. Anal. Math.}, 87:77--101, 2002.
\newblock Dedicated to the memory of Thomas H.\ Wolff.

\bibitem{bbmihes}
J.~Bourgain, H.~Brezis, and P.~Mironescu.
\newblock {$H\sp {1/2}$} maps with values into the circle: minimal connections,
  lifting, and the {G}inzburg-{L}andau equation.
\newblock {\em Publ. Math. Inst. Hautes \'Etudes Sci.}, 99:1--115, 2004.

\bibitem{lddjr}
J.~Bourgain, H.~Brezis, and P.~Mironescu.
\newblock Lifting, degree, and distributional {J}acobian revisited.
\newblock {\em Comm. Pure Appl. Math.}, 58(4):529--551, 2005.

\bibitem{bbmuni}
J.~Bourgain, H.~Brezis, and P.~Mironescu.
\newblock A new function space and applications.
\newblock {\em J. Eur. Math. Soc. (JEMS)}, 17(9):2083--2101, 2015.

\bibitem{bousquetmironescu}
P.~Bousquet and P.~Mironescu.
\newblock {Prescribing the Jacobian in critical Sobolev spaces}.
\newblock {\em J. Anal. Math.}, 122(1):317--373, 2014.


\bibitem{brasseur}
J.~Brasseur.
\newblock {On restrictions of Besov functions}, 2017, preprint, hal-01538362.

\bibitem{brezisfa}
H.~Brezis.
\newblock {\em Functional analysis, {S}obolev spaces and partial differential
  equations}.
\newblock Universitext. Springer, New York, 2011.

\bibitem{blmn}
H.~Brezis, Y.~Li, P.~Mironescu, and L.~Nirenberg.
\newblock Degree and {S}obolev spaces.
\newblock {\em Topol. Methods Nonlinear Anal.}, 13(2):181--190, 1999.

\bibitem{gnp}
H.~Brezis and P.~Mironescu.
\newblock Gagliardo-{N}irenberg, composition and products in fractional
  {S}obolev spaces.
\newblock {\em J. Evol. Equ.}, 1(4):387--404, 2001.
\newblock Dedicated to the memory of Tosio Kato.

\bibitem{brezisnirenberg1}
H.~Brezis and L.~Nirenberg.
\newblock {Degree theory and {BMO}. {I}. {C}ompact manifolds without
  boundaries}.
\newblock {\em {Selecta Math. (N.S.)}}, 1(2):197--263, 1995.



\bibitem{carbou}
G.~Carbou.
\newblock Applications harmoniques \`a\ valeurs dans un cercle.
\newblock {\em C. R. Acad. Sci. Paris S\'er. I Math.}, 314:359--362, 1992.

\bibitem{chemin}
J.-Y. Chemin.
\newblock {\em Perfect incompressible fluids}, volume~14 of {\em Oxford Lecture
  Series in Mathematics and its Applications}.
\newblock The Clarendon Press Oxford University Press, New York, 1998.
\newblock Translated from the 1995 French original by Isabelle Gallagher and
  Dragos Iftimie.

\bibitem{devorepopov}
R.A. DeVore and V.A. Popov.
\newblock {Interpolation of {B}esov spaces}.
\newblock {\em {Trans. Amer. Math. Soc.}}, 305(1):397--414, 1988.

\bibitem{fjw}
M.~Frazier, B.~Jawerth, and G.~Weiss.
\newblock {\em Littlewood-{P}aley theory and the study of function spaces},
  volume~79 of {\em CBMS Regional Conference Series in Mathematics}.
\newblock Published for the Conference Board of the Mathematical Sciences,
  Washington, DC, 1991.

\bibitem{gagliardo}
E.~Gagliardo.
\newblock Caratterizzazioni delle tracce sulla frontiera relative ad alcune
  classi di funzioni in {$n$} variabili.
\newblock {\em Rend. Sem. Mat. Univ. Padova}, 27:284--305, 1957.

\bibitem{isobe2}
T.~Isobe.
\newblock On global singularities of {S}obolev mappings.
\newblock {\em Math. Z.}, 252(4):691--730, 2006.

\bibitem{mazyanew}
V.~Maz'ya.
\newblock {\em Sobolev spaces with applications to elliptic partial
  differential equations}, volume 342 of {\em Grundlehren der Mathematischen
  Wissenschaften [Fundamental Principles of Mathematical Sciences]}.
\newblock Springer, Heidelberg, augmented edition, 2011.

\bibitem{meyer92}
Y.~Meyer.
\newblock {\em {Wavelets and operators. Translated by D. H. Salinger.}}
\newblock Cambridge: Cambridge University Press, 1992.

\bibitem{surveypetru}
P.~Mironescu.
\newblock Sobolev maps on manifolds: degree, approximation, lifting.
\newblock In {\em Perspectives in nonlinear partial differential equations},
  volume 446 of {\em Contemp. Math.}, pages 413--436. Amer. Math. Soc.,
  Providence, RI, 2007.

\bibitem{mironescuphase}
P.~Mironescu.
\newblock {Lifting default for {$\mathbb S^1$}-valued maps}.
\newblock {\em {C. R. Math. Acad. Sci. Paris}}, 346(19-20):1039--1044, 2008.

\bibitem{soreview}
P.~Mironescu.
\newblock {$\Bbb S^1$}-valued {S}obolev mappings.
\newblock {\em Sovrem. Mat. Fundam. Napravl.}, 35:86--100, 2010.

\bibitem{mironescucras2}
P.~Mironescu.
\newblock {Decomposition of {${\mathbb S}^{1}$}-valued maps in Sobolev spaces}.
\newblock {\em C. R. Math. Acad. Sci. Paris}, 348(13-14):743 -- 746, 2010.

\bibitem{tracesoldnew}
P.~Mironescu and E.~Russ.
\newblock {Traces of weighted Sobolev spaces. Old and new}.
\newblock {\em Nonlinear Anal.}, 119:354--381, 2015.

\bibitem{nguyenphase}
H.-M. Nguyen.
\newblock Inequalities related to liftings and applications.
\newblock {\em C. R. Math. Acad. Sci. Paris}, 346:957--962, 2008.

\bibitem{peetre}
J.~Peetre.
\newblock A counterexample connected with {G}agliardo's trace theorem.
\newblock {\em Comment. Math. Special Issue}, 2:277--282, 1979.
\newblock Special issue dedicated to Wladyslaw Orlicz on the occasion of his
  seventy-fifth birthday.

\bibitem{runstsickel}
T.~Runst and W.~Sickel.
\newblock {\em Sobolev spaces of fractional order, {N}emytskij operators, and
  nonlinear partial differential equations}, volume~3 of {\em de Gruyter Series
  in Nonlinear Analysis and Applications}.
\newblock Walter de Gruyter \& Co., Berlin, 1996.

\bibitem{schmeisser}
H.-J. Schmeisser and H.~Triebel.
\newblock {\em Topics in {F}ourier analysis and function spaces}, volume~42 of
  {\em Mathematik und ihre Anwendungen in Physik und Technik [Mathematics and
  its Applications in Physics and Technology]}.
\newblock Akademische Verlagsgesellschaft Geest \& Portig K.-G., Leipzig, 1987.

\bibitem{steinweiss}
E.M. Stein and G.~Weiss.
\newblock {\em Introduction to {F}ourier analysis on {E}uclidean spaces}.
\newblock Princeton University Press, Princeton, N.J., 1971.
\newblock Princeton Mathematical Series, No. 32.

\bibitem{triebel1}
H.~Triebel.
\newblock {\em Interpolation theory, function spaces, differential operators},
  volume~18 of {\em North-Holland Mathematical Library}.
\newblock North-Holland Publishing Co., Amsterdam, 1978.

\bibitem{triebelheat}
H.~Triebel.
\newblock Characterizations of {B}esov-{H}ardy-{S}obolev spaces via harmonic
  functions, temperatures, and related means.
\newblock {\em J. Approx. Theory}, 35(3):275--297, 1982.

\bibitem{triebel2}
H.~Triebel.
\newblock {\em Theory of function spaces}, volume~78 of {\em Monographs in
  Mathematics}.
\newblock Birkh\"auser Verlag, Basel, 1983.

\bibitem{triebel3}
H.~Triebel.
\newblock {\em Theory of function spaces. {II}}, volume~84 of {\em Monographs
  in Mathematics}.
\newblock Birkh\"auser Verlag, Basel, 1992.

\bibitem{triebel06}
H.~Triebel.
\newblock {\em Theory of function spaces. {III}}, volume 100 of {\em Monographs
  in Mathematics}.
\newblock Birkh\"auser Verlag, Basel, 2006.

\bibitem{triebel10}
H.~Triebel.
\newblock {\em Bases in function spaces, sampling, discrepancy, numerical
  integration}, volume~11 of {\em EMS Tracts in Mathematics}.
\newblock European Mathematical Society (EMS), Z\"urich, 2010.

\bibitem{uspenskii}
S.V. Uspenski{\u\i}.
\newblock Imbedding theorems for classes with weights.
\newblock {\em Trudy Mat. Inst. Steklov.}, 60:282--303, 1961.
\newblock English translation: {\it Am. Math. Soc. Transl.}, 87:121--145, 1970.

\bibitem{yamazaki}
M.~Yamazaki.
\newblock A quasihomogeneous version of paradifferential operators. {II}. {A}
  symbol calculus.
\newblock {\em J. Fac. Sci. Univ. Tokyo Sect. IA Math.}, 33(2):311--345, 1986.

\end{thebibliography}


          \end{document}